\documentclass[12pt]{amsart}
\usepackage{amssymb,amsmath,color}
\usepackage{graphicx}
\def\uno{\begin{figure}[htb]\begin{center}\includegraphics[width=10cm]{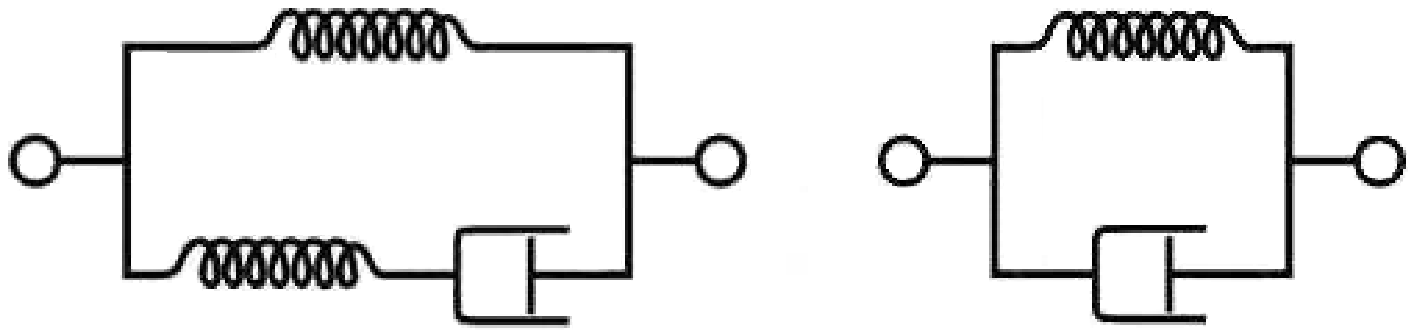}\\
{\tiny fig.\ $\!$1$\,\,$ Mechanical schemes of the standard viscoelastic solid model (left) and
the Kelvin-Voigt model (right)}
\end{center}\end{figure}}
\def\due{\begin{figure}[htb]\begin{center}\includegraphics[width=5cm]{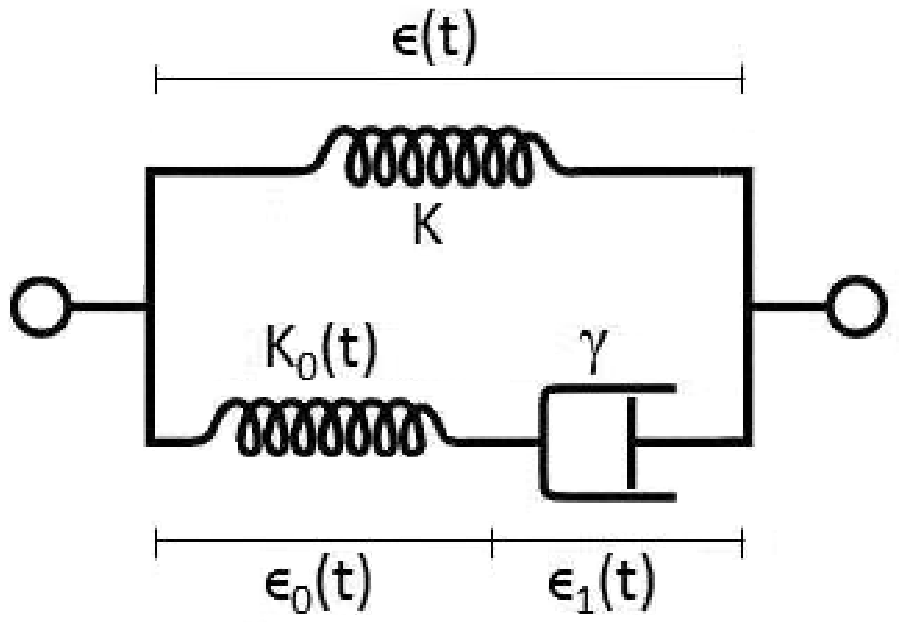}\\
{\tiny fig.\ $\!$2$\,\,$ Mechanical scheme of an aging standard viscoelastic solid}
\end{center}\end{figure}}

\def\eps {\varepsilon}
\def\d {{\rm d}}
\def\e {{\rm e}}
\def\R {\mathbb{R}}
\def\N {\mathbb{N}}
\def\H {{\mathcal H}}
\def\h {{\rm H}}

\def\C {{\mathcal C}}
\def\D {{\mathfrak D}}
\def\Q {{\mathcal Q}}
\def\EE {{\mathcal E}}

\def\T {{\mathbb T}}
\def\L {{\mathcal L}}
\def\M {{\mathcal M}}
\def \l {\langle}
\def \r {\rangle}

\def\ddt{\frac{\d}{\d t}}

\def \and {{\qquad\text{and}\qquad}}

\def\pt{\partial_t}

\def\ps{\partial_s}

\def\ptt{\partial_{tt}}
\def\sp {\text{span}}

\newtheorem{proposition}{Proposition}[section]
\newtheorem{theorem}[proposition]{Theorem}
\newtheorem{corollary}[proposition]{Corollary}
\newtheorem{lemma}[proposition]{Lemma}
\theoremstyle{definition}
\newtheorem{definition}[proposition]{Definition}

\newtheorem{remark}[proposition]{Remark}

\numberwithin{equation}{section}

\def \au {\rm}
\def \ti {\it}
\def \jou {\rm}
\def \bk {\it}
\def \no#1#2#3 {{\bf #1} (#3), #2.}
\def \eds#1#2#3 {#1, #2, #3.}

\title[Viscoelasticity with time-dependent memory kernels]
{A model of viscoelasticity\\
with time-dependent memory kernels}

\author[M. Conti, V. Danese, C. Giorgi and V. Pata]
{Monica Conti, Valeria Danese, Claudio Giorgi and Vittorino Pata}

\address{Politecnico di Milano - Dipartimento di Matematica
\newline\indent
Via Bonardi 9, 20133 Milano, Italy}
\email{monica.conti@polimi.it {\rm (M. Conti)}}
\email{valeria.danese@polimi.it {\rm (V. Danese)}}
\email{vittorino.pata@polimi.it {\rm (V. Pata)}}

\address{Universit\`a di Brescia - DICATAM
\newline\indent
Via Valotti 9, 25133 Brescia, Italy}
\email{giorgi@ing.unibs.it {\rm (C.\ Giorgi)}}



\begin{document}

\begin{abstract}
We consider the model equation arising in the theory of viscoelasticity
$$
\ptt u-h_t(0)\Delta u
-\int_{0}^\infty h_t'(s)\Delta u(t-s)\d s+ f(u) = g.
$$
Here, the main feature is that the memory kernel $h_t(\cdot)$ depends on time, allowing for instance to
describe the dynamics of aging materials. From the mathematical viewpoint, this translates into the study of dynamical systems
acting on time-dependent spaces, according to the newly established theory of Di Plinio {\it et al.}\ \cite{oscillon}.
In this first work, we give a proper notion of solution, and we provide a global well-posedness result.
The techniques naturally extend to the analysis of the longterm behavior of the associated process,
and can be exported to cover the case of general systems with memory in presence of time-dependent kernels.
\end{abstract}

\maketitle

\section{Introduction}

\noindent
Let $\Omega \subset \R^3$ be a bounded domain with smooth boundary
$\partial\Omega$.
For any given $\tau \in \R$, we consider for $t>\tau$ the evolution equation
arising in the theory of uniaxial deformations in isothermal viscoelasticity (see e.g.\ \cite{CHR,FM,RHN})
\begin{equation}
\label{eqn-mem-h}
\ptt u-h(0)\Delta u
-\int_{0}^\infty h'(s)\Delta u(t-s)\d s+ f(u) = g,
\end{equation}
subject to the homogeneous Dirichlet boundary condition
\begin{equation}
\label{bdryC}
u(t)_{|\partial \Omega} = 0.
\end{equation}
The unknown variable
$u = u(x, t): \Omega \times \R \to \R$ describes
the {\it axial displacement field} relative to the reference
configuration of a viscoelastic body occupying the volume
$\Omega$ at rest, and is interpreted as an initial datum for $t\leq\tau$,
where it need not solve the equation.
Here, $f:\R \to \R$ is a nonlinear term,
$g=g(x): \Omega\to \R$ an external force, and
the convolution (or memory) kernel $h$
is a function of the form
$$h(s)=k(s)+k_\infty,$$
where $k$ is a (nonnegative) convex summable function. The values $h(0)>k_\infty>0$
represent the {\it instantaneous elastic modulus},
and the {\it relaxation modulus} of the material, respectively.
Since $h'=k'$, a formal integration by parts yields
$$\int_{0}^\infty h'(s)\Delta u(t-s)\d s
=-k(0)\Delta u(t) + \int_0^\infty k(s)\Delta\pt u(t-s)\d s,$$
so that \eqref{eqn-mem-h} can be rewritten as
\begin{equation}
\label{eqn-mem-c}
\ptt u-k_\infty\Delta u
-\int_{0}^\infty k(s)\Delta\pt u(t-s)\d s+ f(u) = g.
\end{equation}

A simplified, yet very effective, way to represent linear viscoelastic materials is through
rheological models, that is, by considering combinations of linear elastic springs and viscous dashpots.
In particular,
a standard viscoelastic solid is modeled as a Maxwell element, i.e.\ a Hookean spring and a Newtonian dashpot sequentially connected,
which is in parallel with a lone spring.
The resulting memory kernel turns out to be of exponential type.
In this context, the aging of the material corresponds to a change of the physical parameters
along the time leading, possibly, to a different shape of the memory kernel.
There are several ways to reproduce this phenomenon
within a rheological framework (see e.g.\ \cite{Droz}). Here, we propose to describe
aging as a deterioration of the elastic response of the viscoelastic solid,
translating into a progressive stiffening of the spring in the Maxwell element.
In the limiting situation, when the spring
becomes completely rigid, the outcome is the Kelvin-Voigt (solid) model,
depicted by a damper and an elastic spring connected in parallel.

\uno

Mathematically speaking, the Kelvin-Voigt model can be obtained from \eqref{eqn-mem-c}
by keeping fixed the total mass of the kernel $k$, that is,
$$\int_0^\infty k(s)\d s=m,$$
and letting $h(0)\to\infty$. Or, in other words, by taking the ``limit"
$$k(s)\to m\delta_0(s),$$
where $\delta_0$ is the Dirac mass at $0^+$.
This leads to the equation
\begin{equation}
\label{eqn-KV}
\ptt u - k_\infty\Delta u - m\Delta \pt u + f(u) = g.
\end{equation}
In the terminology of Dautray and Lions \cite{DOLI}, this is the passage from viscoelasticity with {\it long memory}
to viscoelasticity with {\it short memory}.

In spite of a relatively vast literature concerning both \eqref{eqn-mem-h} and \eqref{eqn-KV}
(see e.g.\ \cite{CACH1,CACH2,COPA,DAF,FALA,GHMA,GILA,KAL,LIZH,MAS,MRIV,PASQ,PAZE,WE} and references therein),
we are not aware of analytic studies which consider the possibility
of including aging phenomena (or, more generally, changes of the structural properties)
of the material {\it within} the dynamics.
Thus, from our point of view, it is of great interest to have a model
whose physical parameters can evolve over time. This would allow, for instance, to describe the transition from
long to short memory of a given viscoelastic material.

The way to pursue this goal is to let the memory kernel $h$ depend itself on time.
Accordingly, we will consider a modified version of \eqref{eqn-mem-h},
namely,
\begin{equation}
\label{eqn-mem-ht}
\ptt u-h_t(0)\Delta u
-\int_{0}^\infty h_t'(s)\Delta u(t-s)\d s+ f(u) = g,
\end{equation}
subject to the boundary condition \eqref{bdryC}, with
$$h_t(s)=k_t(s)+k_\infty,\quad k_\infty>0,$$
where the time-dependent function $k_t(\cdot)$
is convex and summable for every fixed $t$.
Here and in what follows, the {\it prime} denotes the partial derivative with respect to $s$.
It is worth noting that the nonautonomous character of~\eqref{eqn-mem-ht} is structural, in the sense
that the leading differential operator depends explicitly on time. A much different situation
than, say, having a time-dependent external force.
The equation is supplemented with the initial conditions
\begin{equation}
\label{in-cond-u}
\begin{cases}
u(\tau) = u_\tau,\\
\pt u(\tau) = v_\tau,\\
u(\tau-s) = \phi_\tau(s),\quad s>0,
\end{cases}
\end{equation}
where $u_\tau$, $v_\tau$ and the function $\phi_\tau$ are assigned data.

In order to study the initial-boundary value problem above, following the pioneering idea of Dafermos \cite{DAF,DAF2}
we introduce for $t \geq \tau$ the {\it past history} variable
$$
\eta^t(s) = u(t) - u(t-s), \quad s >0.
$$
Besides, aiming to incorporate the boundary conditions,
we consider the
strictly positive linear operator $A=-\Delta$ on the Hilbert space $L^2(\Omega)$
of square summable functions on $\Omega$, with domain
$$\D(A)=H^2(\Omega)\cap H_0^1(\Omega),$$
where $H_0^1(\Omega)$ and $H^2(\Omega)$ denote the usual Sobolev spaces.
Then, calling
$$\mu_t(s) = -k'_t(s)=-h'_t(s),$$
and setting for simplicity the constant $k_\infty=1$, problem \eqref {eqn-mem-ht}
with the Dirichlet boundary condition \eqref{bdryC} reads
\begin{equation}
\label{eqn-mem}
\ptt u+A u +\int_0^\infty \mu_t(s)A\eta(s)\d s + f(u) = g.
\end{equation}
Denoting
$$\eta_\tau(s) = u_\tau - \phi_{\tau}(s),$$
in view of \eqref{in-cond-u} it is readily seen that,
for every $t \geq \tau$,
\begin{equation}
\label{ETA}
\eta^t(s) =
\begin{cases}
u(t) - u(t-s),                     &s \leq t- \tau,\\
\eta_{\tau}(s-t+\tau) + u(t) - u_\tau, &s > t -\tau.
\end{cases}
\end{equation}
Accordingly, viewing the original problem as the evolution
system \eqref{eqn-mem}-\eqref{ETA} in the variables $u(t)$ and $\eta^t$,
the initial conditions \eqref{in-cond-u} turn into
\begin{equation}
\label{in-cond}
\begin{cases}
u(\tau) = u_\tau, \\
\pt u(\tau) = v_\tau, \\
\eta^\tau = \eta_\tau.
\end{cases}
\end{equation}

The focus of this paper is a global well-posedness result
for problem \eqref{eqn-mem}-\eqref{in-cond} in a suitable functional space. From the mathematical
point of view, the presence of a time-dependent kernel introduces essential difficulties, and new ideas are needed.
Indeed, in the classical Dafermos scheme,
one has a supplementary differential equation ruling the evolution
of the variable $\eta$, generated by the right-translation semigroup
on the history space, whose mild solution is given by~\eqref{ETA}.
But in our case, the natural phase space for the past history is itself time-dependent, suggesting that the right strategy is
to work within the theory of processes on time-dependent spaces $\H_t$,
recently devised by Di Plinio {\it et al.}\ \cite{oscillon}, and further developed in \cite{calimero,katilina,CPT,oscillon2}.
Still, in those papers the time dependence entered only via the definition of the norm
in a universal reference space, i.e.\ the spaces $\H_t$ are in fact the {\it same} linear space
endowed with different norms, all equivalent for $t$ running in compact sets. On the contrary,
here the phase space $\H_t$ depends on time at a {\it geometric} level, and we only have a set inclusion $\H_\tau\subset\H_t$
as $\tau\leq t$.
This
poses some problems even in the definition of the time derivative $\pt\eta$. To overcome this obstacle, we propose
a different notion of solution (which boils down to the usual one when the memory kernel
is time-independent), where the evolution of $\eta$ is actually postulated via the representation formula~\eqref{ETA}.
At the same time, this prevents us to obtain directly differential inequalities,
essential to produce any kind of energy estimates, so that
the main technical tool in our approach turns out to be a family of integral inequalities, which are
obtained by several approximation steps.

The theory, along with the techniques developed in this work, open the way to the longterm analysis of the solutions,
which will be the object of future works. Besides, a paradigm is set
in order to tackle any equation of memory type with time-dependent kernels.
It is worth mentioning also the possibility of extending in a quite natural way the underlying ideas
to the study of systems with memory in the so-called minimal state framework introduced in~\cite{FGP}.
\medskip

\subsection*{Outline of the paper}

In Section 2 we stipulate our assumptions on the time-dependent memory kernel, showing a concrete example
of physical relevance, while in Section 3 we introduce the proper functional spaces. The
global well-posedness result is stated in Section 4.
The main technical tool needed in our analysis is discussed in Section 5, and
the remaining of the paper is devoted to the proofs: existence of solutions (Section 6), uniqueness (Section 7)
and further regularity (Section 8). In the final Appendix we provide a physical
derivation of our equation via a rheological model for aging materials.

\subsection*{Notation}

For $\sigma \in \R$, we define the compactly nested Hilbert spaces
$$\h^\sigma = \D(A^{\sigma/2}),$$
endowed with the inner products and norms
$$\l u,v\r_{\sigma}
=\l A^{\sigma/2}u,A^{\sigma/2}v\r_{L^2(\Omega)}
\and \|u\|_{\sigma}
=\|A^{\sigma/2}u\|_{L^2(\Omega)}.$$
The index $\sigma$ will be always omitted when equal to zero.
For $\sigma>0$, it is understood that $\h^{-\sigma}$
denotes the completion of the domain, so that $\h^{-\sigma}$ is the dual space of $\h^\sigma$.
The symbol $\l\cdot,\cdot\r$ will also be used to denote the duality pairing between
$\h^{-\sigma}$ and $\h^\sigma$.
In particular,
$$\h^{-1}=H^{-1}(\Omega),\quad \h=L^2(\Omega),\quad
\h^1=H_0^1(\Omega), \quad
\h^2=H^2(\Omega)\cap H_0^1(\Omega).$$
Along the paper, we will repeatedly use without explicit mention the Young, H\"older and
Poincar\'e inequalities,
as well as the standard Sobolev embeddings, e.g.\  $\h^1\subset L^6(\Omega)$.

\section{The Time-Dependent Memory Kernel}

\subsection{General assumptions}
\label{first-ker}
In order to prove a well-posedness result for our problem,
we suppose that the function
$$(t,s)\mapsto\mu_t(s):\R\times\R^+\to\R^+$$
satisfies the following set of assumptions, where $\R^+=(0,\infty)$ and
we agree to denote
$$\dot\mu_t(s)=\pt\mu_t(s)\and \mu'_t(s)=\ps\mu_t(s),$$
whenever such derivatives exist.

\smallskip
\begin{itemize}
\item[{\bf (M1)}]
For every fixed $t\in\R$, the map
$s\mapsto\mu_t(s)$ is nonincreasing,
absolutely continuous and summable.

\smallskip
\item[{\bf (M2)}]
For every $\tau\in\R$ there exists a function $K_\tau:[\tau,\infty)\to\R^+$,
summable on any interval $[\tau,T]$,
such that
$$
\mu_t(s) \leq K_\tau(t)\mu_\tau(s)
$$
for every $t\geq\tau$ and every $s>0$.

\smallskip
\item[{\bf (M3)}]
For almost every fixed $s>0$,
the map $t\mapsto\mu_t(s)$ is differentiable for all $t\in\R$.
Besides,
$$(t,s)\mapsto \mu_t(s)\in L^\infty({\mathcal K})\and (t,s)\mapsto \dot\mu_t(s)\in L^\infty({\mathcal K})$$
for every compact set ${\mathcal K}\subset \R\times\R^+$.

\smallskip
\item[{\bf (M4)}]
There exists a function $M:\R\to\R^+$,
bounded on bounded intervals, such that
$$
\dot\mu_t(s)+\mu_t'(s)\leq M(t)\mu_t(s)
$$
for every $t\in\R$ and almost every $s>0$.
\end{itemize}

\medskip
Here are some immediate consequences of the assumptions. First,
due to {\bf (M1)}, the function $s\mapsto\mu_t(s)$ is differentiable almost
everywhere and, for every $t\in\R$,
$$
\mu_t'(s)\leq 0,\quad\text{for a.e. }s>0.
$$
Note that $s\mapsto\mu_t(s)$ can be
possibly unbounded in a neighborhood of zero. Besides,
denoting the total mass of $\mu_t$ by
$$\kappa(t) = \int_0^\infty \mu_t(s)\d s,
$$ from {\bf (M2)} we readily see that
\begin{equation}
\label{kappa-t-tau}
\kappa(t) \leq K_\tau(t)\kappa(\tau),
\quad\forall t\geq\tau.
\end{equation}

\begin{remark}
In this work we are mainly concerned with kernels that do not vanish on $\R^+$, modeling
the so-called infinite delay case. However, our analysis applies as well (and with no changes in the proofs)
to the finite delay case, namely, when
$$s_\infty(t)=\sup\{s>0:\, \mu_t(s)>0\}<\infty.$$
In this case, in comply with {\bf (M2)}, note that $s_\infty(t)$ is a nonincreasing function of $t$.
\end{remark}

\subsection{A concrete example: the rescaled kernel}
An enlightening example of the kind of kernel we have in mind is obtained by a
suitable rescaling  of a (nonnegative) nonincreasing
function $\mu\in \C^1(\R^+)\cap L^1(\R^+)$.
More precisely, given  $\eps\in\C^1(\R,\R^+)$ satisfying
$$
\dot\eps(t)\leq 0, \quad\forall t\in\R,
$$
we define
$$\mu_t(s)
=\frac{1}{[\eps(t)]^2}\,\mu\bigg(\frac{s}{\eps(t)}\bigg).$$

\begin{remark}
With reference to \eqref{eqn-mem-c},
defining
$$k(s)=\int_s^\infty \mu(y)\d y,$$
we have
$$k_t(s)=\frac{1}{\eps(t)}\,k\bigg(\frac{s}{\eps(t)}\bigg).$$
In particular, if we assume that $s\mapsto s\mu(s)\in L^1(\R^+)$, we get
$$
\int_0^\infty k_t(s)\d s=\int_0^\infty k(s)\d s=\int_0^\infty s\mu(s)\d s=m<\infty.
$$
Accordingly, if $\eps(t)\to 0$ as $t\to \infty$, we obtain the distributional convergence
$$\lim_{ t\to \infty}k_t= m \delta_0.$$
\end{remark}

We now verify that such a $\mu_t$ complies with the assumptions above.
We begin to write explicitly the derivatives of $\mu_t$, namely,
$$
\mu'_t(s)
=\frac{1}{[\eps(t)]^3}\mu'\bigg(\frac{s}{\eps(t)}\bigg),
$$
and
$$
\dot\mu_t(s)
=-\frac{\dot\eps(t)}{\eps(t)}\,\big[2\mu_t(s)+s\mu'_t(s)\big].
$$

\noindent
\medskip
$\bullet$ Assumptions {\bf (M1)} and {\bf (M3)} are obviously verified. In particular,
we have that
$$\kappa(t)=
\int_0^\infty\mu_t(s)\d s
=\frac{1}{\eps(t)}\int_0^\infty\mu(s)\d s<\infty.$$

\noindent
\smallskip
$\bullet$ Assumption {\bf (M2)} holds with
$$K_\tau(t)= \bigg[\frac{\eps(\tau)}{\eps(t)}\bigg]^2.$$
This easily follows from the fact that both $\mu$ and $\eps$ are nonincreasing
in the respective arguments.

\noindent
\medskip
$\bullet$ Assumption {\bf (M4)} holds with
$$M(t)=-2\,\frac{\dot\eps(t)}{\eps(t)}.$$
Indeed,
$$
\dot\mu_t(s)+\mu_t'(s)=\Big[1-s\frac{\dot\eps(t)}{\eps(t)}\Big]\mu'_t(s)-2\,\frac{\dot\eps(t)}{\eps(t)}\mu_t(s).
$$
Since $\dot\eps\leq 0$ and $\mu'_t\leq 0$, we obtain the desired inequality.

\begin{remark}
The typical (and physically relevant) example is obtained by taking
$$\mu(s)=k(s)=\e^{-s},$$
in which case
$$k_t(s)=\frac{1}{\eps(t)}\,\e^{-\frac{s}{\eps(t)}}
\and
\mu_t(s)=\frac{1}{[\eps(t)]^2}\,\e^{-\frac{s}{\eps(t)}}.$$
\end{remark}

\section{Time-Dependent Memory Spaces}

\noindent
Let $\sigma\in\R$ and $\tau\in\R$ be arbitrarily fixed.
For every $t\geq\tau$, we introduce the memory spaces
$$\M_t^{\sigma} = L^2_{\mu_t}(\R^+; \h^{\sigma+1}),$$
equipped with the weighted $L^2$-inner products
$$\l\eta,\xi\r_{\M_t^{\sigma}}
=\int_0^\infty \mu_t(s)\l\eta(s),\xi(s)\r_{\sigma+1}\d s.$$
Owing to {\bf (M2)}, for every $\eta \in \M_\tau^{\sigma}$ we have
\begin{equation}
\label{norme-Mt-tau}
\|\eta\|^2_{\M_t^{\sigma}}
\leq K_\tau(t)\|\eta\|^2_{\M_\tau^{\sigma}},
\quad\forall t \geq \tau,
\end{equation}
providing the continuous embedding
$$\M_\tau^{\sigma} \subset \M_t^{\sigma},\quad\forall t\geq\tau.$$
We will also consider the linear operator
$$\T_t: \D(\T_t)\subset \M_t^\sigma \to \M_t^\sigma,$$
acting as
$$\T_t\eta=-\eta',$$
the {\it prime} standing for weak derivative,
with domain
$$\D(\T_{t})=\Big\{\eta\in{\M^\sigma_t}:\,
\eta'\in\M^\sigma_t,\,\,
\lim_{s\to 0}\eta(s)=0\,\,\,\text{in}\,\,\, \h^{\sigma+1}\Big\}.$$
It is well known (see e.g.\ \cite{Terreni}) that $\T_t$ is
the infinitesimal generator of the contraction semigroup of right-translations on the space
$\M_t^\sigma$,  hence a dissipative operator. More precisely, we have the estimate
\begin{equation}
\label{T-diss}
\l \T_t\eta, \eta\r_{\M_t^{\sigma}}
= \frac12\int_0^\infty \mu'_t(s) \|\eta(s)\|^2_{\sigma+1}\d s,\quad\forall \eta\in\D(\T_t),
\end{equation}
which by {\bf (M1)} readily yields
$$\l \T_t\eta, \eta\r_{\M_t^{\sigma}}\leq 0,\quad\forall \eta\in\D(\T_t).$$
Due to \eqref{norme-Mt-tau}, we also observe that
\begin{equation}
\label{inkp}
\T_t \supset \T_\tau, \quad\forall t \geq \tau.
\end{equation}
In fact, the operators $\{\T_t\}_{t\geq\tau}$ are
increasingly nested extensions of each other.
Finally, we define the {\it extended memory spaces}
$$\H^\sigma_t = \h^{\sigma+1} \times \h^\sigma \times \M_t^{\sigma},$$
with the usual product norm
$$\|(u,v,\eta)\|_{\H^\sigma_t}^2=\|u\|_{\sigma+1}^2+\|v\|_\sigma^2+\|\eta\|_{\M_t^{\sigma}}^2.$$
Again, the subscript $\sigma$ is omitted whenever zero.

\section{Statement of the Result}
\label{sec-mainresults}

\noindent
In this section, we give the definition of a (weak) solution to our problem, and we state the
main existence and uniqueness result.
We first stipulate the assumptions on the external force $g$ and on the nonlinearity $f$.

\smallskip
\noindent
$\bullet$ Let $g\in \h$.

\smallskip
\noindent
$\bullet$ Let  $f\in\C^1(\R)$, with $f(0)=0$, satisfy the growth restriction
\begin{equation}
\label{hp1-f}
|f'(u)| \leq C\big(1+|u|^2),
\end{equation}
for some $C\geq 0$, along with the dissipation condition
\begin{equation}
\label{phi1}
\liminf_{|u|\to\infty}\frac{f(u)}{u}>-\lambda_1,
\end{equation}
$\lambda_1>0$ being the first eigenvalue of $A$.

\begin{definition}
\label{def-sol}
Let $T>\tau \in \R$, and let $z_\tau=(u_\tau,v_\tau,\eta_\tau)\in\H_\tau$ be a fixed vector.
A function
$$z(t) = (u(t), \pt u(t), \eta^t)\in\H_t \quad\text{for a.e. } t\in [\tau,T]$$
is a solution
to problem \eqref{eqn-mem}-\eqref{in-cond}
on the time-interval $[\tau, T]$ with initial datum  $z_\tau$ if:
\smallskip
\begin{itemize}
\item[{\rm (i)}] $u \in L^\infty(\tau, T; \h^1)$, $\,\pt u\in  L^\infty(\tau, T; \h)$, $\,\ptt u\in  L^1(\tau, T; \h^{-1})$.
\smallskip
\item[{\rm (ii)}] $u(\tau)=u_\tau$, $\,\pt u(\tau)=v_\tau$.
\smallskip
\item[{\rm (iii)}] The function $\eta$ fulfills the representation formula \eqref{ETA}.
\smallskip
\item[{\rm (iv)}] For every test function $\varphi \in \h^1$ and almost every $t \in [\tau,T]$,
$$
\l \ptt u(t),\varphi\r + \l u(t),\varphi\r_1
+ \int_0^\infty \mu_t(s)\l \eta^t(s), \varphi\r_1 \d s
+ \l f(u(t)),\varphi \r = \l g,\varphi\r.
$$
\end{itemize}
\end{definition}

\begin{remark}
As it will be shown in the proofs, the facts that $u_\tau\in\h^1$, $\eta_\tau\in\M_\tau$
and $u \in L^\infty(\tau, T; \h^1)$ are enough to guarantee that
$\eta^t$ given by~\eqref{ETA} belongs to $\M_t$ for almost every $t\in[\tau,T]$.
\end{remark}

\begin{remark}
By means of standard embeddings (see e.g\ \cite[\S5.9]{Evans}), point (i) of
the definition yields at once
$$
u \in \C([\tau, T], \h) \cap \C^1([\tau, T], \h^{-1}).
$$
Thus, speaking of the initial values of $u$ and $\pt u$ makes sense.
\end{remark}

\begin{remark}
As already mentioned in the Introduction,
it is worth noting that the definition above, where the representation formula \eqref{ETA}
is actually postulated,
is applicable as well to classical systems with memory (i.e.\ in presence of time-independent kernels),
providing a notion of solution completely equivalent to the usual one (see e.g.\ \cite{COPA,PZ}).
In fact, this approach seems to be even more natural, and considerably simplifies
the proofs of existence and uniqueness results. In particular,
it allows to avoid cumbersome regularization arguments, needed to justify
certain formal multiplications (cf.\ \cite{PZ}).
\end{remark}

Within the assumptions above on $\mu_t$, $g$ and $f$, we can state
our well-posedness theorem.

\begin{theorem}
\label{thm-ex-un}
For every $T>\tau \in \R$ and every initial datum
$z_\tau = (u_\tau, v_\tau, \eta_\tau) \in \H_\tau$,
problem \eqref{eqn-mem}-\eqref{in-cond} admits a unique solution $z(t)=(u(t), \pt u(t), \eta^t)$
on the interval $[\tau, T]$.
Besides,
\begin{align*}
& u  \in \C([\tau, T], \h^1) \cap \C^1([\tau, T], \h),\\
& \eta^t \in\M_t,\quad\forall t\in[\tau,T],
\end{align*}
and
$$\sup_{t\in[\tau,T]}\|z(t)\|_{\H_t}<C,$$
for some $C>0$ depending only on $T,\tau$ and the size of the initial datum.
\end{theorem}

Actually, given any two solutions $z_1(t)$ and $z_2(t)$ on $[\tau,T]$,
the following continuous dependence result holds.

\begin{theorem}
\label{thm-cont-2}
There exists a positive constant $C$, depending only on $T,\tau$ and the size of the initial data, such that
$$\| z_1(t) - z_2(t) \|_{\H_t}
\leq C\| z_1(\tau) - z_2(\tau) \|_{\H_\tau}$$
for every $t \in [\tau,T]$.
\end{theorem}

Then, for every initial datum $z_\tau\in\H_\tau$, we can write the solution $z(t)$ as
$$z(t)=U(t,\tau)z_\tau.$$
The two-parameter family of operators
$$U(t,\tau):\H_\tau\to\H_t,\quad t\geq\tau,$$
is called a {\it processes on time-dependent spaces}
(see \cite{calimero,katilina,CPT,oscillon,oscillon2}), characterized by the two properties:
\begin{itemize}
\smallskip
\item[(i)] $U(\tau,\tau)$ is the identity map on ${\H_\tau}$ for every $\tau$;
\smallskip
\item[(ii)] $U(t,\tau)U(\tau,s)=U(t,s)$ for every $t\geq\tau\geq s$.
\smallskip
\end{itemize}

The rest of the paper is devoted to the proofs of Theorems \ref{thm-ex-un}
and \ref{thm-cont-2}.

\section{A Key Inequality}
\label{sec-key}

\noindent
The main technical tool in order to produce the energy estimates
needed in the analysis
is an integral inequality involving the norm of the auxiliary variable
in the time-dependent memory space.
Let then $\sigma\in\R$ and $T>\tau\in \R$ be arbitrarily fixed, and let
$$u \in W^{1,\infty}(\tau,T;\h^{\sigma+1})\and \eta_\tau \in \M^\sigma_\tau$$
be any two given functions.
Recall the standard embedding
$$W^{1,\infty}(\tau,T; \h^{\sigma+1})\subset\C([\tau,T],\h^{\sigma+1}).$$
Defining $\eta=\eta^t(s)$, with $(t,s)\in [\tau,T]\times\R^+$, by the formula
\begin{equation}
\label{eta-rappr}
\eta^t(s) =
\begin{cases}
u(t) - u(t-s),                     &\, s \leq t- \tau,\\
\eta_{\tau}(s-t+\tau) + u(t) - u(\tau), & \, s > t -\tau,
\end{cases}
\end{equation}
the following theorem holds.

\begin{theorem}
\label{theorem-eta-norm}
For all $\tau\leq a\leq b\leq T$,
we have the inequality
$$
\|\eta^b\|^2_{\M_b^{\sigma}}
\leq \|\eta^a\|^2_{\M_a^{\sigma}}
+M\int_a^b\|\eta^t\|^2_{\M_t^{\sigma}}\d t
+2\int_a^b\l\pt u(t),\eta^t\r_{\M_t^{\sigma}}\d t,
$$
having set
$$M=\sup_{t\in[\tau,T]}M(t).$$
\end{theorem}

The proof of Theorem~\ref{theorem-eta-norm} requires a number of preparatory lemmas.

\begin{lemma}
\label{lemma-bd-eta}
Setting
$$\Phi(u,\eta_\tau)=6\kappa(\tau)\|u\|^2_{L^\infty(\tau,T;\h^{\sigma+1})}+3\|\eta_{\tau}\|^2_{\M_\tau^{\sigma}},$$
we have that $\eta^t \in \M_\tau^{\sigma}\subset \M_t^{\sigma}$ with
$$\|\eta^t\|_{\M_\tau^\sigma}^2\leq \Phi(u,\eta_\tau),\quad \forall t\in[\tau,T],$$
and
$$\|\eta^t\|_{\M_t^\sigma}^2\leq \Phi(u,\eta_\tau)K_\tau(t)\in L^1(\tau,T).$$
\end{lemma}

\begin{proof}
Recalling that $\mu_\tau(\cdot)$ is nonincreasing, we have
\begin{align*}
\|\eta^t\|^2_{\M_\tau^\sigma}
&=\int_0^{t-\tau}\mu_\tau(s)\|u(t)-u(t-s)\|^2_{\sigma+1}\d s \\
&\quad+\int_{t-\tau}^\infty\mu_\tau(s)
\|\eta_{\tau}(s-t+\tau)+u(t)-u(\tau)\|^2_{\sigma+1}\d s \\
&\leq 6\|u\|^2_{L^\infty(\tau,T;\h^{\sigma+1})}\int_0^\infty \mu_\tau(s)\d s
+3\int_0^\infty\mu_\tau(s+t-\tau)
\|\eta_{\tau}(s)\|^2_{\sigma+1}\d s\\
\noalign{\vskip1mm}
&\leq \Phi(u,\eta_\tau).
\end{align*}
The latter inequality follows from {\bf (M2)} and \eqref{norme-Mt-tau}.
\end{proof}

\begin{remark}
\label{rem-bd-eta}
It is clear from the proof that the conclusion of Lemma \ref{lemma-bd-eta} is true without any assumption on $\pt u$,
provided that $u\in\C([\tau,T],\h^{\sigma+1})$.
\end{remark}

\begin{lemma}
\label{lemma-eta-norm0}
Assume in addition that
$\eta_\tau\in\D(\T_\tau)$. Then $\eta^t\in \D(\T_\tau)$ for every $t\in [\tau,T]$,
$\eta\in W^{1,\infty}(\tau,T; \M_\tau^\sigma)$ and  the differential equation
$$\pt\eta^t = \T_\tau\eta^t + \pt u(t)$$
holds in $\M_\tau^{\sigma}$.
\end{lemma}

\begin{proof}
Since $\eta_\tau\in \D(\T_\tau)\subset\M_\tau^\sigma$
and $u\in W^{1,\infty}(\tau,T; \h^{\sigma+1})$,
we can differentiate \eqref{eta-rappr} with respect to $s$ and to $t$ in the weak sense,
so obtaining
\begin{equation}
\label{T-eta}
\ps\eta^t(s) =
\begin{cases}
\pt u(t-s),                  &s \leq t-\tau,\\
\eta'_\tau(s-t+\tau),        &s > t-\tau,
\end{cases}
\end{equation}
and
\begin{equation}
\label{pt-eta}
\pt\eta^t(s) =
\begin{cases}
\pt u(t)-\pt u(t-s),                  &s \leq t-\tau,\\
\pt u(t)-\eta'_\tau(s-t+\tau),       &s > t-\tau.
\end{cases}
\end{equation}
Let us prove that $\eta^t \in \D(\T_\tau)$.
Since $u\in \C([\tau,T],\h^{\sigma+1})$,
we readily obtain the limit
$$\lim_{s\to 0} \eta^t(s)=0 \quad\text{in }\h^{\sigma+1}.$$
Moreover, as $\mu_\tau(\cdot)$ is nonincreasing
and $\eta_\tau \in \D(\T_\tau)\subset\M_\tau^\sigma$,
\begin{align}
\label{medpseta}
\|\ps\eta^t\|^2_{\M_\tau^\sigma}
&=\int_0^{t-\tau}\mu_\tau(s)\|\pt u(t-s)\|^2_{\sigma+1}\d s
+ \int_{t-\tau}^\infty\mu_\tau(s)\|\eta'_\tau(s-t+\tau)\|^2_{\sigma+1}\d s\\
&\leq \kappa(\tau)\|\pt u\|_{L^\infty(\tau,T; \h^{\sigma+1})}^2 + \|\eta'_\tau\|^2_{\M_\tau^\sigma},\notag
\end{align}
yielding $\ps\eta^t \in \M_\tau^\sigma$.
Analogous calculations provide the estimate
$${\rm ess}\sup_{\hskip-5mm t\in[\tau,T]}\|\pt\eta^t\|_{\M^\sigma_\tau}<\infty,$$
which, together with Lemma~\ref{lemma-bd-eta}, gives
$\eta\in W^{1,\infty}(\tau,T; \M_\tau^\sigma)$.
Finally, collecting \eqref{T-eta} and \eqref{pt-eta}, it follows that
the differential equation
$$\pt\eta^t = \T_\tau\eta^t + \pt u(t)$$
holds in $\M_\tau^{\sigma}$.
\end{proof}

\begin{remark}
Since $\M_\tau^{\sigma} \subset \M_t^{\sigma}$, recalling \eqref{inkp} the latter differential equation gives
\begin{equation}
\label{pippopazzo}
\pt\eta^t = \T_t\eta^t + \pt u(t),
\end{equation}
where the equality holds in $\M_t^\sigma$ at any fixed $t$.
\end{remark}

\begin{remark}
When  $\eta_\tau\in \D(\T_\tau)$, from \eqref{norme-Mt-tau} and \eqref{medpseta} we
deduce the estimate
\begin{equation}
\label{salva}
\|\ps \eta^t\|_{\M_t^\sigma}^2\leq
\Psi(u,\eta_\tau)
K_\tau(t),\quad\forall t\in[\tau,T],
\end{equation}
where
$$\Psi(u,\eta_\tau)=\kappa(\tau)\|\pt u\|^2_{L^\infty(\tau,T;\h^{\sigma+1})}+\|\eta'_{\tau}\|^2_{\M_\tau^{\sigma}}.
$$
\end{remark}

We are ready to prove an integral inequality for more regular data.

\begin{lemma}
\label{lemma-eta-norm}
Assume that $u \in \C^1([\tau,T],\h^{\sigma+1})$ and
$\eta_\tau \in \C^1(\R^+,\h^{\sigma+1})\cap\D(\T_\tau)$.
Then, for all $\tau\leq a\leq b\leq T$,
we have the inequality
$$
\|\eta^b\|^2_{\M_b^{\sigma}}
\leq \|\eta^a\|^2_{\M_a^{\sigma}}
+\int_a^b\!\!\int_0^\infty\big[\dot\mu_t(s)+\mu'_t(s)\big]\|\eta^t(s)\|^2_{\sigma+1}\d s\,\d t
+2\int_a^b\l\pt u(t),\eta^t\r_{\M_t^{\sigma}}\d t.
$$
\end{lemma}

\begin{proof}
For every $\eps>0$ small,
we introduce the cut-off function
$$\phi_\eps(s)=
\begin{cases}
0 & \text{if } 0\leq s<\eps,\\
s/\eps-1 & \text{if } \eps\leq s<2\eps,\\
1 & \text{if } 2\eps\leq s\leq 1/\eps,\\
2-\eps s & \text{if }  1/\eps\leq s<2/\eps,\\
0  & \text{if } 2/\eps\leq s.
\end{cases}
$$
Correspondingly, we define the family of approximate kernels
$$\mu_t^\eps(s)=\phi_\eps(s)\mu_t(s).$$
Denoting now
$$f_\eps(t,s)=\mu^\eps_t(s)\|\eta^t(s)\|^2_{\sigma+1},$$
we claim that
\begin{equation}
\label{puppo}
\int_0^\infty\ddt f_\eps(t,s)\d s=\ddt\int_0^\infty f_\eps(t,s)\d s.
\end{equation}
Indeed, from Lemma~\ref{lemma-bd-eta} we know that $s\mapsto f_\eps(t,s)\in L^1(\R^+)$ for every fixed $t$.
Moreover, since $t\mapsto\|\eta^t(s)\|^2_{\sigma+1}\in\C^1([\tau,T])$ for every $s$,
$$\ddt f_\eps(t,s)=\dot\mu^\eps_t(s)\|\eta^t(s)\|^2_{\sigma+1}+ 2\mu^\eps_t(s)\l\pt\eta^t(s),\eta^t(s)\r_{\sigma+1},$$
and \eqref{puppo} follows once we show the bound
\begin{equation}
\label{puppos}
\int_0^\infty\sup_{t\in [\tau,T]}\left|\ddt f_\eps(t,s)\right|\d s<\infty.
\end{equation}
To this end, in light of the assumptions on $u$ and $\eta_\tau$ along with formulae \eqref{eta-rappr} and \eqref{pt-eta},
we note that
$$\sup_{t\in [\tau,T]}\sup_{s\in [\eps,2/\eps]}\Big[\|\eta^t(s)\|_{\sigma+1}+\|\partial_t\eta^t(s)\|_{\sigma+1}\Big]<\infty.$$
Hence, exploiting {\bf (M3)} on the compact set ${\mathcal K}=[\tau,T]\times [\eps,2/\eps]$, there exists $C_\eps>0$
such that
$$
\left|\ddt f_\eps(t,s)\right|\leq C_\eps\phi_\eps(s)\leq C_\eps\chi_{[\eps,2/\eps]}(s).
$$
This proves~\eqref{puppos}.

\smallskip
\noindent
At this point, introducing the $\eps$-dependent memory space
$$\M_t^{\sigma,\eps}=L^2_{\mu_t^\eps}(\R^+;\h^{\sigma+1}),$$
with the usual scalar product and norm,
we multiply~\eqref{pippopazzo} by $2\eta^t$ in $\M_t^{\sigma,\eps}$, so to get
$$2\l \pt\eta^t,\eta^t\r_{\M_t^{\sigma,\eps}} =2\l \T_t\eta^t,\eta^t\r_{\M_t^{\sigma,\eps}} + 2\l \pt u(t),\eta^t\r_{\M_t^{\sigma,\eps}}.$$
Making use of~\eqref{puppo},
\begin{align*}
2\l\pt\eta^t,\eta^t\r_{\M_t^{\sigma,\eps}}
&= \int_0^\infty \mu^\eps_t(s)\ddt\|\eta^t(s)\|^2_{\sigma+1}\d s \\
&= \int_0^\infty\bigg[\ddt\big(\mu^\eps_t(s)\|\eta^t(s)\|^2_{\sigma+1}\big)
- \dot\mu^\eps_t(s)\|\eta^t(s)\|^2_{\sigma+1}\bigg]\d s \\
&= \ddt\|\eta^t\|^2_{\M_t^{\sigma,\eps}}
- \int_0^\infty \dot\mu^\eps_t(s)\|\eta^t(s)\|^2_{\sigma+1}\d s.
\end{align*}
Besides, from \eqref{T-diss} applied in the space $\M_t^{\sigma,\eps}$,
$$2\l \T_t\eta^t,\eta^t\r_{\M_t^{\sigma,\eps}}=\int_0^\infty (\mu^\eps_t)'(s)
\|\eta^t(s)\|^2_{\sigma+1}\d s.$$
In summary, we end up with
$$\ddt \|\eta^t \|^2_{\M_t^{\sigma,\eps}}=\int_0^\infty\big[\dot\mu^\eps_t(s)+(\mu^\eps_t)'(s)\big]
\|\eta^t(s)\|^2_{\sigma+1}\d s
+2\l\pt u(t),\eta^t\r_{\M_t^{\sigma,\eps}}.$$
As a byproduct of~\eqref{puppo}-\eqref{puppos}, we also infer that the map $t\mapsto \|\eta^t\|^2_{\M_t^{\sigma,\eps}}$
is absolutely continuous. This allows us to integrate the differential identity above, obtaining
\begin{align}
\label{INT-eps}
\|\eta^b \|^2_{\M_b^{\sigma,\eps}}-\|\eta^a \|^2_{\M_a^{\sigma,\eps}}&
-\int_a^b\int_0^\infty\big[\dot\mu^\eps_t(s)+(\mu^\eps_t)'(s)\big]
\|\eta^t(s)\|^2_{\sigma+1}\d s\,\d t\\\nonumber
&=2\int_a^b\l\pt u(t),\eta^t\r_{\M_t^{\sigma,\eps}}\d t.
\end{align}
In order to complete the proof, it suffices to pass to the limit in \eqref{INT-eps} as $\eps\to 0$.
Note first that, for any fixed $t$,
$$0\leq \|\eta^t \|^2_{\M_t^{\sigma}}-\|\eta^t \|^2_{\M_t^{\sigma,\eps}}
\leq\int_{0}^{2\eps} \mu_t(s)\|\eta^t(s)\|^2_{\sigma+1}\d s
+\int_{1/\eps}^\infty \mu_t(s)\|\eta^t(s)\|^2_{\sigma+1}\d s\to 0.$$
Analogously, for any fixed $t$ we verify that
$$\l\pt u(t),\eta^t\r_{\M_t^{\sigma,\eps}}\to\l\pt u(t),\eta^t\r_{\M_t^{\sigma}}.$$
Exploiting {\bf (M2)} and Lemma~\ref{lemma-bd-eta},
\begin{align*}
|\l\pt u(t),\eta^t\r_{\M_t^{\sigma,\eps}}|
&\leq \sqrt{\kappa(t)}\|\pt u(t)\|_{\sigma+1}\|\eta^t\|_{\M_t^{\sigma}}\\
&\leq \sqrt{\kappa(\tau)}\|\pt u(t)\|_{\sigma+1}\sqrt{K_\tau(t)}\|\eta^t\|_{\M_t^{\sigma}}\in L^1(a,b),
\end{align*}
and the Dominated Convergence Theorem entails
$$\int_a^b\l\pt u(t),\eta^t\r_{\M_t^{\sigma,\eps}}\d t\to \int_a^b\l\pt u(t),\eta^t\r_{\M_t^{\sigma}}\d t.$$
Thus,
denoting
\begin{align*}
g_\eps(t,s) &=-\big[\dot\mu^\eps_t(s)+(\mu^\eps_t)'(s)\big]\|\eta^t(s)\|^2_{\sigma+1},\\
g(t,s) &=-\big[\dot\mu_t(s)+\mu_t'(s)\big]\|\eta^t(s)\|^2_{\sigma+1},
\end{align*}
we are left to prove that
$$
\int_a^b\!\!\int_0^\infty g(t,s)\d s\,\d t
\leq \liminf_{\eps\to 0}\int_a^b\!\!\int_0^\infty g_\eps(t,s)\d s\,\d t.
$$
Indeed, in light of {\bf (M4)},
\begin{align*}
g_\eps(t,s)
&=-\big[\phi_\eps(s)\dot\mu_t(s)+\phi_\eps(s)\mu_t'(s)+\phi_\eps'(s)\mu_t(s)\big]\|\eta^t(s)\|^2_{\sigma+1}\\
&\geq -M(t)\mu_t(s)\|\eta^t(s)\|^2_{\sigma+1}-\frac{1}{\eps}\chi_{[\eps,2\eps]}(s)\mu_t(s)\|\eta^t(s)\|^2_{\sigma+1}.
\end{align*}
We infer from Lemma \ref{lemma-bd-eta} that
the first term in the right-hand side above belongs to $L^1((a,b)\times \R^+)$. Concerning the second one,
we observe that
$$
\|\eta^t(s)\|^2_{\sigma+1}\leq \Big(\int_0^s\|\partial_s \eta^t(y)\|_{\sigma+1}\d y\Big)^2
\leq s \int_0^s\|\partial_s \eta^t(y)\|_{\sigma+1}^2\d y,
$$
implying in turn, as $\mu_t(\cdot)$ is nonincreasing,
$$
\mu_t(s)\|\eta^t(s)\|^2_{\sigma+1}\leq s\int_0^s\mu_t(y)\|\partial_s \eta^t(y)\|_{\sigma+1}^2\d y
\leq s\|\ps \eta^t\|_{\M_t^\sigma}^2 \leq \Psi(u,\eta_\tau) sK_\tau(t),
$$
where \eqref{salva} is invoked in the last passage.
Besides, since we can assume $\eps\leq 1$,
$$\frac{s}{\eps}\chi_{[\eps,2\eps]}(s)\leq 2\chi_{[0,2]}(s).
$$
Collecting the two inequalities above, we end up with
$$\frac{1}{\eps}\chi_{[\eps,2\eps]}(s)\mu_t(s)\|\eta^t(s)\|^2_{\sigma+1}\leq
2\Psi(u,\eta_\tau)\chi_{[0,2]}(s)K_\tau(t)\in L^1((a,b)\times \R^+).$$
In conclusion, we found a (positive) function
$$\psi(t,s)=M(t)\mu_t(s)\|\eta^t(s)\|^2_{\sigma+1}
+2\Psi(u,\eta_\tau)\chi_{[0,2]}(s)K_\tau(t)\in L^1((a,b)\times\R^+)$$
satisfying
$$g_\eps(t,s)\geq -\psi(t,s).$$
We are in a position to apply Fatou's Lemma:
since $g_\eps(t,s)\to g(t,s)$ almost everywhere, the required inequality follows.
\end{proof}

By {\bf (M4)} we have a straightforward corollary.

\begin{corollary}
\label{cor-eta-norm}
Within the hypotheses of Lemma~\ref{lemma-eta-norm}, for all $\tau\leq a\leq b\leq T$,
we have the inequality
$$
\|\eta^b\|^2_{\M_b^{\sigma}}
\leq \|\eta^a\|^2_{\M_a^{\sigma}}
+M\int_a^b\|\eta^t\|^2_{\M_t^{\sigma}}\d t
+2\int_a^b\l\pt u(t),\eta^t\r_{\M_t^{\sigma}}\d t.
$$
\end{corollary}

\begin{proof}[Proof of Theorem \ref{theorem-eta-norm}]
Choose two sequences
$$\eta_{\tau n} \in \C^1(\R^+, \h^{\sigma+1})\cap\D(\T_\tau)
\and u_n\in  \C^1([\tau,T],\h^{\sigma+1})$$
such that
\begin{align*}
\eta_{\tau n} \to \eta_\tau & \,\,\,\text{ in }\M_\tau^{\sigma},\\
u_n\to u & \,\,\,\text{ in } W^{1,\infty}(\tau,T;\h^{\sigma+1}),
\end{align*}
and define $\eta_n=\eta_n^t(s)$ as
$$
\eta_n^t(s) =
\begin{cases}
u_n(t) - u_n(t-s),                     &\, s \leq t- \tau,\\
\eta_{\tau n}(s-t+\tau) + u_n(t) - u_n(\tau), & \, s > t -\tau.
\end{cases}
$$ From Corollary~\ref{cor-eta-norm},
we know that
$$
\|\eta_n^b\|^2_{\M_b^{\sigma}}
\leq \|\eta_n^a\|^2_{\M_a^{\sigma}}
+M\int_a^b\|\eta_n^t\|^2_{\M_t^{\sigma}}\d t
+2\int_a^b\l\pt u_n(t),\eta_n^t\r_{\M_t^{\sigma}}\d t.
$$
All is needed is passing to the limit in the inequality above.
By means of Lemma~\ref{lemma-bd-eta} applied to the difference $\eta_n-\eta$ and
to $\eta_n$, we draw the estimate
$$\|\eta^t_n-\eta^t\|_{\M_t^\sigma}^2\leq \Phi(u_n-u,\eta_{\tau n}-\eta_\tau)K_\tau(t),$$
implying the pointwise convergence
$$\eta^t_n\to\eta^t\,\,\,\text{ in }\M_t^\sigma,\quad\forall t\in[a,b],$$
along with the control
$$\|\eta^t_n\|_{\M_t^\sigma}^2\leq \sup_n\Phi(u_n,\eta_{\tau n})\,K_\tau(t)\in L^1(a,b).$$
In particular,
$$\|\eta^t_n \|^2_{\M^{\sigma}_t}\to \|\eta^t \|^2_{\M^{\sigma}_t},\quad\forall t\in[a,b],$$
and, by the Dominated Convergence Theorem,
$$
\int_a^b\|\eta_n^t\|^2_{\M_t^{\sigma}}\d t\to\int_a^b\|\eta^t\|^2_{\M_t^{\sigma}}\d t.
$$
In order to establish the remaining convergence
$$\int_a^b\l\pt u_n(t),\eta_n^t\r_{\M_t^{\sigma}}\d t\to \int_a^b\l\pt u(t),\eta^t\r_{\M_t^{\sigma}}\d t,$$
we argue as in the proof of Lemma \ref{lemma-eta-norm}. Indeed,
$$\l\pt u_n(t),\eta_n^t\r_{\M_t^{\sigma}}\to \l\pt u(t),\eta^t\r_{\M_t^{\sigma}},\quad\text{for a.e. } t\in[a,b],$$
and
$$|\l\pt u_n(t),\eta_n^t\r_{\M_t^{\sigma}}|\leq \sqrt{\kappa(t)}\|\pt u_n(t)\|_{\sigma+1}\|\eta_n^t\|_{\M_t^{\sigma}}
\leq C K_\tau(t)\in L^1(a,b),$$
having set
$$C=\sup_n\Big[\sqrt{\kappa(\tau)\Phi(u_n,\eta_{\tau n})}\,\|\pt u_n\|_{L^\infty(\tau,T; \h^{\sigma+1})}\Big].$$
A further application of the Dominated Convergence Theorem will do.
This finishes the proof of Theorem~\ref{theorem-eta-norm}.
\end{proof}

\section{Existence of Solutions}
\label{sec-existence}

\noindent
We are now ready to prove the existence result.

\begin{theorem}
\label{thm-ex-un-wak}
For every $T>\tau \in \R$ and every initial datum
$z_\tau = (u_\tau, v_\tau, \eta_\tau) \in \H_\tau$,
problem \eqref{eqn-mem}-\eqref{in-cond} admits at least a solution $z(t)$
on $[\tau, T]$. Moreover, $z(t)\in\H_t$ for every $t$ and
$$\sup_{t\in[\tau,T]}\|z(t)\|_{\H_t}\leq C,$$
for some $C>0$ depending only on $T,\tau$ and the size of the initial datum.
\end{theorem}

The proof of the theorem is based on a Galerkin procedure,
where the first step consists in looking for smooth solutions to suitable approximating problems on
finite-dimensional spaces.

\subsection{Galerkin approximations}
Let $\{w_n\}$ be an orthonormal basis of $\h$ which is also orthogonal
in $\h^1$.
For every $n \in \N$, we define the finite-dimensional subspace
$$\h_n = \sp\{ w_1, \dots, w_n \}\subset \h^1$$
and we denote by $P_n:\h\to\h_n$ the orthogonal projection onto $\h_n$.
We approximate the initial datum
$z_\tau = (u_\tau, v_\tau, \eta_\tau)$ with a sequence
$z_{\tau n} = (u_{\tau n}, v_{\tau n}, \eta_{\tau n})$, where
\begin{align}
\label{proj-un}
u_{\tau n} & = P_n u_\tau\to u_\tau \!\quad\text{in }\h^1, \\
\label{proj-vn}
v_{\tau n} & = P_n v_\tau\to v_\tau \quad\text{in }\h, \\
\label{proj-etan}
\eta_{\tau n} & = P_n\eta_\tau \to \eta_\tau \quad\text{in }\M_\tau.
\end{align}
For every $n \in \N$, we look for $T_n \in (\tau,T]$ and
$$u_n : [\tau,T_n] \to \h_n$$
satisfying, for every test function $\varphi\in \h_n$ and every $t \in [\tau,T_n]$,
\begin{equation}
\label{weak-eqn-approx}
\l \ptt u_n(t),\varphi\r + \l u_n(t),\varphi\r_1
+ \int_0^\infty \mu_t(s)\l\eta^t_n(s),\varphi\r_1 \d s
+ \l f(u_n(t)),\varphi\r = \l g,\varphi\r,
\end{equation}
where
\begin{equation}
\label{eta-n-rappr}
\eta^t_n(s) =
\begin{cases}
u_n(t) - u_n(t-s),                          &s \leq t- \tau,\\
\eta_{\tau n}(s-t+\tau) + u_n(t) - u_{\tau n}, &s > t -\tau,
\end{cases}
\end{equation}
along with the initial conditions
\begin{equation}
\label{in-un}
\begin{cases}
u_n(\tau)=u_{\tau n},\\
\pt u_n(\tau)=v_{\tau n}.
\end{cases}
\end{equation}

\begin{lemma}
\label{l-galerkin}
For every $n \in \N$, there exist $T_n \in (\tau,T]$ and
a pair $(u_n,\eta_n)$ satisfying
\eqref{weak-eqn-approx}-\eqref{in-un}, where
$u_n$ is of the form
$$
u_n(t) = \sum_{j=1}^n a_j^n(t)w_j,\quad
a_j^n\in \C^2([\tau,T_n]).$$
\end{lemma}

The proof is completely standard, and therefore omitted.
It is enough to note that \eqref{weak-eqn-approx} translates into
a system of $n$ integro-differential equations
in the unknowns $a_j^n$, and the existence (and uniqueness) of a local solution is guaranteed
by a classical ODEs result, owing to the fact that the nonlinearity $f$ is locally Lipschitz.

\medskip
According to Lemma \ref{l-galerkin}, we denote by
$$z_n(t) = (u_n(t), \pt u_n(t), \eta_n^t)$$ the (local) solution to the approximating problem at time $t$.
In what follows, $C$ will denote a generic positive constant and $\Q: \R^+ \to \R^+$ a generic nondecreasing
positive function,
both (possibly) depending only on $\tau$, $T$ and the structural parameters of the problem,
but {\it independent} of $n$.

\subsection{Energy estimates}
The crucial step is finding suitable a priori estimates
for the approximate solution $z_n$.

\begin{lemma}
\label{lemma-bdd-n-1}
Let $\|z_{\tau}\|_{\H_\tau}\leq R$ for some $R\geq 0$. Then
$z_n(t)\in\H_t$ for every $t$ and
$$
\sup_{t\in[\tau,T]}\|z_n(t)\|_{\H_t}\leq \Q(R).
$$
\end{lemma}

\begin{proof}
We preliminarily observe that,
owing to \eqref{proj-un}-\eqref{proj-etan},
\begin{equation}
\label{prpr}
\|z_{\tau n}\|_{\H_\tau}\leq \|z_\tau\|_{\H_\tau}\leq R.
\end{equation}
For $t \in [\tau,T_n]$, we define the energy functional
$$E_n(t)=\|u_n(t)\|^2_1+\|\pt u_n(t)\|^2+ 2\l F(u_n(t)),1\r,$$
where
$$F(u)=\int_0^u f(s)\d s.$$
Exploiting \eqref{hp1-f},
$$2\l F(u_n),1\r \leq C\big(1+\|u_n\|_1^4\big).$$
Besides, condition~\eqref{phi1}
implies that
$$
2\l F(u_n),1\r\geq -(1-\theta)\|u_n\|^2_1-C,
$$
for some $0<\theta<1$.
Thus we have the two-side control
\begin{equation}
\label{contr-E-Z-n}
\theta\big[\|u_n(t)\|^2_1+\|\pt u_n(t)\|^2\big]-C\leq E_n(t)
\leq \Q(\|u_n(t)\|_1+\|\pt u_n(t)\|).
\end{equation}
Testing \eqref{weak-eqn-approx} with
$\varphi=\pt u_n$, we draw the equality
$$\ddt E_n + 2\l\eta_n, \pt u_n\r_{\M_t} = 2\l g,\pt u_n\r.$$
Since by \eqref{contr-E-Z-n}
$$2\l g,\pt u_n\r\leq 2\|g\|\|\pt u_n\|\leq C+CE_n,$$
an integration on $[\tau, t]$ with $t<T_n$ yields
$$
E_n(t)+2\int_\tau^t \l\eta_n^y,\pt u_n(y)\r_{\M_y}\d y\leq E_n(\tau)+C+C\int_\tau^t E_n(y)\d y.
$$
Knowing that
$u_n \in W^{1,\infty}(\tau,T_n;\h^1)$,
$\eta_{\tau n} \in \M_\tau$
and $\eta_n$ fulfills \eqref{eta-n-rappr}, we are allowed to apply
Theorem \ref{theorem-eta-norm} for $\sigma=0$, so obtaining
$$\|\eta_n^t\|^2_{\M_t}\leq \|\eta_n^\tau\|^2_{\M_\tau} + M
\int_\tau^t\|\eta_n^y\|^2_{\M_y}\d y
+2\int_\tau^t \l\eta_n^y,\pt u_n(y)\r_{\M_y}\d y.$$
Therefore, setting
$$\EE_n(t)=E_n(t)+\|\eta_n^t\|^2_{\M_t}$$
and adding the latter two integral inequalities, using again \eqref{contr-E-Z-n} we end up with
$$\EE_{n}(t) \leq  \EE_n(\tau)+C+C\int_\tau^t \EE_n(y)\d y.$$
The claim follows from the Gronwall Lemma and a further application of \eqref{contr-E-Z-n},
together with \eqref{prpr}.
\end{proof}

Since the estimates for $(u_n,\pt u_n, \eta_n)$ do not depend on $n$,
we conclude that the solutions to the approximate problems are global, namely,
$$T_n=T,\quad\forall n\in\N.$$

\subsection{Passage to the limit} From Lemma~\ref{lemma-bdd-n-1}
we learn that
$$
u_n \,\,\,\text{is bounded in }\, L^\infty(\tau,T;\h^1)\cap W^{1,\infty}(\tau,T;\h).
$$
Hence, there exists $u \in L^\infty(\tau,T;\h^1)\cap W^{1,\infty}(\tau,T;\h)$
such that, up to a subsequence,
\begin{align}
\label{conv-un-u}
u_n \xrightarrow{{\rm w}^*} u & \quad\text{in } L^\infty(\tau,T;\h^1),\\
\label{conv-dot-un-v}
\pt u_n \xrightarrow{{\rm w}^*} \pt u & \quad\text{in } L^\infty(\tau,T;\h).
\end{align}
By the classical Simon-Aubin compact embedding~\cite{SIM}
$$L^\infty(\tau,T;\h^1) \cap W^{1,\infty}(\tau,T;\h) \Subset \C([\tau,T],\h),$$
we deduce (up to a further subsequence)
\begin{equation}
\label{uco}
u_n\to u\quad\text{in } \C([\tau,T],\h),
\end{equation}
along with the pointwise convergence
$$u_n \to u \quad\text{a.e. in } [\tau,T] \times \Omega.$$
Thanks to the continuity of $f$, this also yields
\begin{equation}
\label{lif}
f(u_n) \to f(u) \quad\text{a.e. in } [\tau,T] \times \Omega.
\end{equation}
At this point, having $u$ and $\eta_\tau$, we merely define
the function $\eta^t$ for $t \in [\tau,T]$ by \eqref{ETA}.

\begin{remark}
\label{remetareggy}
Since $u\in L^\infty(\tau,T;\h^1)$ and $u_\tau\in\h^1$, recasting word by word the proof of Lemma \ref{lemma-bd-eta}
we find the bound
\begin{equation}
\label{etataub}
\eta\in L^\infty(\tau,T;\M_\tau).
\end{equation}
In turn, we infer from \eqref{norme-Mt-tau} that
$\eta^t\in\M_t$ for almost every $t\in [\tau,T]$.
\end{remark}

\begin{lemma}
The function $z(t)=(u(t), \pt u(t), \eta^t)$ fulfills point {\rm (iv)} of Definition \ref{def-sol}.
\end{lemma}
\begin{proof}
Let $\varphi \in \h_m$ be fixed. Then, for every $n \geq m$,
we have
$$\l\ptt u_n,\varphi\r + \l u_n,\varphi\r_1
+ \int_0^\infty \mu_t(s)\l \eta_n(s), \varphi\r_1 \d s
+ \l f(u_n),\varphi \r = \l g,\varphi\r.$$
Multiplying the above equality by an arbitrary $\zeta \in \C^\infty_{\rm c}([\tau,T])$
and integrating on the interval $[\tau,T]$
we are led to
\begin{align*}
&-\int_\tau^T\dot\zeta( t)\l\pt u_n( t),\varphi\r\d  t
+ \int_\tau^T\zeta( t)\l u_n( t),\varphi\r_1\d  t\\
&\quad+\int_\tau^T\zeta( t)\int_0^\infty\mu_ t(s)\l\eta^ t_n(s),\varphi\r_1\d s\,\d  t
+ \int_\tau^T\zeta( t)\l f(u_n( t)),\varphi\r\d  t\\
&=\l g,\varphi\r\int_\tau^T\zeta( t)\d  t.
\end{align*}
We claim that we can pass to the limit in this equality, getting
\begin{align}
\label{pttl}
&-\int_\tau^T\dot\zeta( t)\l\pt u( t),\varphi\r\d  t
+ \int_\tau^T\zeta( t)\l u( t),\varphi\r_1\d  t\\ \nonumber
&\quad+\int_\tau^T\zeta( t)\int_0^\infty\mu_ t(s)\l\eta^ t(s),\varphi\r_1\d s\,\d  t
+ \int_\tau^T\zeta( t)\l f(u( t)),\varphi\r\d  t\\ \nonumber
&=\l g,\varphi\r\int_\tau^T\zeta( t)\d  t.
\end{align}
Owing to the density of $\h_m$ in $\h^1$ as $m\to\infty$, this finishes the proof of the lemma.

\smallskip
Coming to the claim, we see that the only nontrivial terms to control are the
nonlinear one containing $f(u_n)$ and
\begin{equation}
\label{dite}
\int_\tau^T\zeta( t)\int_0^\infty\mu_ t(s)\l\eta_n^ t(s),\varphi\r_1\d s\,\d  t.
\end{equation}
Concerning the first, the convergence to the corresponding one with $f(u)$
follows by observing that
$$
f(u_n) \xrightarrow{{\rm w}} f(u) \quad\text{in }L^2(\tau,T;\h).
$$
Indeed, by the growth condition \eqref{hp1-f} and Lemma \ref{lemma-bdd-n-1}
$$
\|f(u_n)\|\leq C\big(1+\|u_n\|^3_1\big)\leq \Q(R),
$$
and the result is a consequence of the Weak Dominated Convergence Theorem, in light
of the pointwise convergence~\eqref{lif}.

We are left to pass \eqref{dite} to the limit. To this aim, we
set
$$\bar u_{\tau n} = u_{\tau n}-u_\tau, \qquad \bar \eta _{\tau n} = \eta_{\tau n}-\eta_\tau,$$
and, for every $t \in [\tau,T]$,
\begin{align*}
\bar u_n(t) = u_n(t) - u(t), \qquad \bar\eta^t_n = \eta^t_n - \eta^t.
\end{align*}
Besides, we consider the map
$p_n : [\tau,T] \to \R$ defined as
$$t \mapsto p_n(t) = \l \bar u_n(t), \varphi\r_1,$$
In light of \eqref{conv-un-u},
\begin{equation}
\label{conv-w-st-pn}
p_n \xrightarrow{{\rm w}^*} 0
\quad\text{in } L^\infty(\tau,T).
\end{equation}
Writing explicitly $\bar\eta^t_n$
as
\begin{equation}
\label{repetabar}
\bar\eta^t_n(s) =
\begin{cases}
\bar u_n(t) - \bar u_n(t-s),                     &s \leq t- \tau,\\
\bar\eta_{\tau n}(s-t+\tau) + \bar u_n(t) - {\bar u}_{\tau n}, &s > t -\tau,
\end{cases}
\end{equation}
we have
\begin{align*}
&\int_\tau^{T}\zeta(t) \int_0^\infty\mu_t(s)\l\bar\eta^t_n(s),\varphi\r_1\d s\,\d t\\
&=\int_\tau^{T}\zeta(t)\kappa(t)p_n(t)\d t
-\int_\tau^{T}\zeta(t)\int_0^{t-\tau}\mu_t(s)p_n(t-s)\d s\,\d t\\
&\quad+ \int_\tau^{T}\zeta(t) \int_{t-\tau}^\infty\mu_t(s)\l\bar\eta_{\tau n}(s-t+\tau) - {\bar u}_{\tau n},\varphi\r_1\d s\,\d t.
\end{align*}
It is easy to see that the first term in the right-hand side goes to zero. Indeed, by {\bf (M2)} and \eqref{kappa-t-tau},
$$\int_\tau^{T}|\zeta(t)\kappa(t)|\d t
\leq C\kappa(\tau)\int_\tau^{T}K_\tau(t)\d t\leq C,$$
and \eqref{conv-w-st-pn} readily gives
$$\int_\tau^{T}\zeta(t)\kappa(t)p_n(t)\d t\to 0.$$
Concerning the second term,
an application of the Fubini Theorem yields
\begin{align*}
\int_\tau^{T}\zeta(t)\int_0^{t-\tau}\mu_t(s) p_n(t-s)\d s\,\d t
&=\int_\tau^{T} \zeta(t)\int_\tau^t\mu_t(t-s)p_n(s)\d s\,\d t \\
&=\int_\tau^{T} p_n(s)\int_s^{T}\mu_t(t-s)\zeta(t)\d t\,\d s.
\end{align*}
Appealing again to the Fubini Theorem
and exploiting {\bf (M2)}, we obtain
\begin{align*}
\int_\tau^{T} \left|\int_s^{T}\mu_t(t-s)\zeta(t)\d t\right|\d s
&\leq\int_\tau^{T}|\zeta(t)|\int_\tau^{t}\mu_t(t-s)\d s\,\d t\\
&\leq C\int_\tau^{T}\int_0^{t-\tau}\mu_t(s)\d s\,\d t\\
&\leq C\kappa(\tau)\int_\tau^{T}K_\tau(t)\d t \leq C,
\end{align*}
and \eqref{conv-w-st-pn} ensures the convergence
$$
\int_\tau^{T}\zeta(t)\int_0^{t-\tau}\mu_t(s)p_n(t-s)\d s\,\d t \to 0.
$$
Finally, recalling that $\mu_t$ is nonincreasing,
owing to {\bf (M2)} and using
\eqref{proj-un} and \eqref{proj-etan}, we draw
\begin{align*}
&\left|\int_\tau^{T}\zeta(t)\int_{t-\tau}^\infty\mu_t(s)\l\bar\eta_{\tau n}(s-t+\tau)
- {\bar u}_{\tau n},\varphi\r_1\d s\,\d t\right| \\
&\leq C\|\varphi\|_1 \int_\tau^{T}\int_{0}^\infty\mu_t(s+t-\tau)
\big[\|\bar\eta_{\tau n}(s)\|_1
+ \|{\bar u}_{\tau n}\|_1\big]\d s\,\d t \\
&\leq C\bigg[\int_{0}^\infty \mu_\tau(s)\|\bar\eta_{\tau n}(s)\|_1\d s
+ \kappa(\tau)\|\bar u_{\tau n}\|_1\bigg]\int_\tau^{T}K_\tau(t)\d t \\
\noalign{\vskip2mm}
&\leq C
\big[\sqrt{\kappa(\tau)}\|\bar\eta_{\tau n}\|_{\M_\tau}
+ \|\bar u_{\tau n}\|_1\big]\to 0.
\end{align*}
In summary,
$$
\int_\tau^{T} \zeta(t)\int_0^\infty \mu_t(s)\l \bar\eta^t_n(s),\varphi\r_1\d s\, \d t\to 0,
$$
which completes the proof of the claim.
\end{proof}

\subsection{Regularity}
\label{SSREG}
We already know that $u\in L^\infty(\tau, T; \h^1)$, $\pt u\in L^\infty(\tau, T; \h)$,
$\eta^t\in\M_t$ for almost every $t\in [\tau,T]$.
In particular, $z(t)\in \H_t$ for almost every $t\in [\tau,T]$.
In order to comply with Definition \ref{def-sol}, we are left to verify that
$$\ptt u\in  L^1(\tau, T; \h^{-1}).$$
We need a useful observation.

\begin{lemma}
\label{lemma-L1}
Let $\sigma\in\R$ and $T>\tau\in \R$ be arbitrarily fixed.
If  $\eta\in L^\infty(\tau,T;\M^\sigma_\tau)$, then the map
$$t\mapsto\int_0^\infty\mu_t(s)A\eta^t(s)\d s\in L^1(\tau,T;H^{\sigma-1}).$$
\end{lemma}

\begin{proof}
A simple computation yields
\begin{align*}
\int_\tau^T \Big\|\int_0^\infty\mu_t(s)A\eta^t(s)\d s\Big\|_{\sigma-1}\d t
&\leq\int_\tau^T\int_0^\infty\mu_t(s)\|\eta^t(s)\|_{\sigma+1}\d s\,\d t\\
&\leq\int_\tau^T K_\tau(t)\int_0^\infty\mu_\tau(s)\|\eta^t(s)\|_{\sigma+1}\d s\,\d t\\
&\leq\sqrt{\kappa(\tau)}\|\eta\|_{L^\infty(\tau,T;\M^\sigma_\tau)}\int_\tau^T K_\tau(t)\d t,
\end{align*}
and the thesis follows from {\bf (M2)}.
\end{proof}

In light of \eqref{etataub}, by applying Lemma \ref{lemma-L1}
for $\sigma=0$, the claimed regularity for $\ptt u$ is obtained by comparison in \eqref{pttl}.
As a byproduct, we deduce the continuity
$$\pt u\in \C([\tau,T],\h^{-1}).$$

\subsection{Initial values}
Here we show that the initial conditions are fulfilled, i.e.\
\begin{equation}
\label{in-cond-bis}
u(\tau) = u_\tau
\and
\pt u(\tau) = v_\tau.
\end{equation}
We take any $\varphi \in \h^1$ and $\zeta \in \C^2([\tau,T])$ satisfying
$\zeta(T)=\pt\zeta(T)=0$. Observing that
$$\int_\tau^T\zeta(t)\l\ptt u(t),\varphi\r\d t=-\zeta(\tau)\l\pt u(\tau),\varphi\r
+\dot\zeta(\tau)\l u(\tau),\varphi\r+\int_\tau^T\ddot\zeta(t)\l u(t),\varphi\r\d t,$$
we obtain
\begin{align}
\label{eqn-in}
&-\zeta(\tau)\l\pt u(\tau),\varphi\r
+\dot\zeta(\tau)\l u(\tau),\varphi\r
+\int_\tau^T\ddot\zeta(t)\l u(t),\varphi\r\d t
+ \int_\tau^T\zeta(t)\l u(t),\varphi\r_1\d t\\ \nonumber
&\quad
+\int_\tau^T\zeta(t)\int_0^\infty\mu_t(s)
\l\eta^t(s),\varphi\r_1\d s\,\d t
+ \int_\tau^T\zeta(t)\l f(u(t)),\varphi\r\d t\\ \nonumber
&= \int_\tau^T\zeta(t)\l g,\varphi\r\d t.
\end{align}
On the other hand, arguing in a similar manner with the approximate problem
\eqref{weak-eqn-approx},
\begin{align*}
&-\zeta(\tau)\l \pt u_n(\tau),\varphi\r
+\dot\zeta(\tau)\l u_n(\tau),\varphi\r
+\int_\tau^T\ddot\zeta(t)\l u_n(t),\varphi\r\d t
+ \int_\tau^T\zeta(t)\l u_n(t),\varphi\r_1\d t\\
&\quad
+\int_\tau^T\zeta(t)\int_0^\infty\mu_t(s)
\l\eta^t_n(s),\varphi\r_1\d s\,\d t
+ \int_\tau^T\zeta(t)\l f(u_n(t)),\varphi\r\d t\\
&= \int_\tau^T\zeta(t)\l g,\varphi\r\d t.
\end{align*}
Passing to the limit in the latter identity
and comparing the limiting equality with \eqref{eqn-in} yields
$$\zeta(\tau)\l \pt u(\tau),\varphi\r
-\dot\zeta(\tau)\l u(\tau),\varphi\r
=\zeta(\tau)\l v_\tau,\varphi\r
-\dot\zeta(\tau)\l u_\tau,\varphi\r.$$
Being $\zeta(\tau)$ and $\dot\zeta(\tau)$ arbitrarily chosen,
\eqref{in-cond-bis} holds.
\qed

\subsection{Uniform estimates}
To complete the proof of Theorem \ref{thm-ex-un-wak}, we have to prove that
$z(t)\in\H_t$ for every $t\in[\tau,T]$ and
$$\sup_{t\in[\tau,T]}\|z(t)\|_{\H_t}\leq \Q(R),$$
whenever $\|z_\tau\|_{\H_\tau}\leq R$.
This is obtained by passing to the limit in the uniform estimate of Lemma \ref{lemma-bdd-n-1}.
Due to the convergence $(u_n,\pt u_n) \xrightarrow{{\rm w}}(u,\pt u)$ in $\h^1\times \h$ (at any fixed $t$),
together with the (weak) continuity $u \in \C([\tau, T], \h) \cap \C^1([\tau, T], \h^{-1})$, which allows us to select
the continuous representative in the equivalence classes of $u$ and $\pt u$,
we have that $(u(t),\pt u(t))\in \h^1\times \h$ for every $t\in [\tau, T]$ and
$$
\sup_{t\in[\tau,T]}\Big[\|u(t)\|_1+\|\pt u(t)\|\Big]\leq \Q(R).
$$
The only difficult part is showing that $\eta^t\in \M_t$ for every $t\in [\tau, T]$ and
$$
\sup_{t\in[\tau,T]}\|\eta^t\|_{\M_t}\leq \Q(R).
$$
For every fixed $t\in [\tau,T]$, Lemma \ref{lemma-bdd-n-1} provides the convergence (up to a subsequence)
$$\eta_n^t \xrightarrow{w} q^t  \quad\text{in } \M_t,$$
for some $q^t\in\M_t$. Accordingly,
$$\|q^t\|_{\M_t}\leq \liminf_{n\to\infty}\|\eta_n^t\|_{\M_t}\leq \Q(R).$$
Consequently, if we prove the equality
$q^t=\eta^t$ in $\M_t$ we are done. To see that, it is enough to show that
$$\eta_n^t\to \eta^t\quad\text{in }\M_t^{-1}.$$
But, since $\bar u_n\in \C([\tau,T],\h)$, this follows by applying Lemma~\ref{lemma-bd-eta}
and the subsequent Remark~\ref{rem-bd-eta} for $\sigma=-1$
to the difference
$\bar\eta_n^t=\eta_n^t-\eta^t$ given
by formula \eqref{repetabar}, yielding
$$
\|\bar \eta^t\|^2_{\M_t^{-1}}
\leq \Phi(\bar u_n,\bar\eta_{\tau n})K_\tau(t)\to 0.
$$
Indeed, $\bar u_n\to 0$ in $\C([\tau,T],\h)$ by \eqref{uco},
and $\bar\eta_{\tau n}\to 0$ in $\M_\tau$ by construction.
\qed

\smallskip
The proof of Theorem \ref{thm-ex-un-wak} is completed.

\section{Uniqueness}
\label{sec-unique}

\noindent
Uniqueness is an immediate consequence of the
following weak continuous dependence.

\begin{proposition}
\label{lemma-dep}
Let $z_1(t)=(u_1(t),\pt u_1(t),\eta_1^t)$ and $z_2(t)=(u_2(t),\pt u_2(t),\eta_2^t)$ be any two solutions on $[\tau,T]$.
There exists a positive constant $C$, depending only on $T,\tau$ and the size of the initial data in $\H_\tau$, such that
$$\| z_1(t) - z_2(t) \|_{\H_t^{-1}}
\leq C\| z_1(\tau) - z_2(\tau) \|_{\H_\tau}$$
for every $t \in [\tau,T]$.
\end{proposition}

\begin{proof}
Thanks to Theorem \ref{thm-ex-un-wak},
\begin{equation}
\label{u-bdd}
\sup_{t\in [\tau,T]}\big[\|u_1(t)\|_1+\|u_2(t)\|_1\big]\leq C,
\end{equation}
where here and along the proof, $C>0$ will stand for a generic constant
(possibly) depending on $T,\tau$ and the size of the initial data in $\H_\tau$.
For $t\in [\tau,T]$, we denote by
$$\bar{z}(t)=(\bar{u}(t),\pt\bar{u}(t),\bar{\eta}^t)=z_1(t)-z_2(t)$$
the difference of the two solutions, and we set
$$\Lambda(t)=\|\bar u(t)\|^2+\|\pt\bar u(t)\|^2_{-1}.$$
For every $\varphi \in \h^1$
and almost every $t \in [\tau,T]$ we have
\begin{equation}
\label{eq-mem-diff}
\l\ptt\bar u(t),\varphi\r+ \l\bar{u}(t),\varphi\r_1
+ \int_0^\infty \mu_t(s)\l\bar{\eta}^t(s),\varphi\r_1\d s
+ \l f(u_1(t))-f(u_2(t)),\varphi\r=0,
\end{equation}
with
$$
\bar\eta^t(s) =
\begin{cases}
\bar u(t) - \bar u(t-s),                     &s\leq t-\tau,\\
\bar\eta_\tau(s-t+\tau)+\bar u(t)-{\bar u}_{\tau}, &s> t-\tau.
\end{cases}
$$
Using $\varphi=A^{-1}\pt\bar u(t)$ as a test function in \eqref{eq-mem-diff},
we obtain
$$\ddt \Lambda+
2\int_0^\infty\mu_t(s)\l\bar\eta(s),\pt\bar u\r\d s
=2\l f(u_2)-f(u_1),\pt\bar u\r_{-1}.$$
Exploiting \eqref{hp1-f} and the uniform boundedness
\eqref{u-bdd},
we have the estimate
\begin{align*}
\|f(u_2)-f(u_1)\|_{-1}
&\leq C\| \big(1+|u_1|^2 + |u_2|^2\big)|\bar u| \|_{L^{6/5}(\Omega)} \\
&\leq C\big(1+\|u_1\|_1^2 + \|u_2\|_1^2)\|\bar u\|\\
&\leq C\|\bar u\|.
\end{align*}
Accordingly,
$$2\l f(u_2)-f(u_1),\pt\bar u\r_{-1}\leq C\|\bar u\|\|\pt\bar u\|_{-1}\leq C\Lambda,$$
and we arrive at
$$
\ddt \Lambda(t)+ 2\l \bar\eta^t,\pt\bar u(t)\r_{\M^{-1}_t}
\leq C\Lambda(t).$$
An integration on $[\tau,t]$, with $t\leq T$, entails
$$
\Lambda(t)+2\int_\tau^{t}\l \bar\eta^y,\pt\bar u(y)\r_{\M^{-1}_y}\d y
\leq \Lambda(\tau)+C\int_\tau^{t}\Lambda(y)\d y.
$$
Since $\bar u\in W^{1,\infty}(\tau,T;\h)$ by Theorem~\ref{thm-ex-un-wak}, we
can apply Theorem~\ref{theorem-eta-norm} to $\bar\eta$ for $\sigma=-1$, to get
$$
\|\bar\eta^{t}\|^2_{\M_{t}^{-1}}
- \|\bar\eta^\tau\|^2_{\M_\tau^{-1}}
\leq M\int_\tau^t \|\bar\eta^y\|^2_{\M_y^{-1}}\d y
+2\int_\tau^{t}\l\bar\eta^y,\pt \bar u(y)\r_{\M_y^{-1}}\d y.
$$
Adding the two inequalities, we end up with
$$
\|\bar z(t)\|^2_{\H_t^{-1}}
\leq \|\bar z(\tau)\|^2_{\H_\tau^{-1}}+C\int_\tau^{t}\|\bar z(y)\|^2_{\H_y^{-1}}\d y,
$$
and the conclusion follows from the Gronwall Lemma and the embedding
$\H_\tau \subset \H_\tau^{-1}$.
\end{proof}

\section{Time Continuity and Continuous Dependence}
\label{sec-contdep}

\noindent
To complete our program, we are left to prove the continuity in time of the
solution, along with the strong continuous dependence estimate of Theorem \ref{thm-cont-2}.
The proofs of both results are obtained by approximating the
solutions originating from fixed initial data in $\H_\tau$
with smoother solutions departing from more regular data.

\subsection{Two preliminary lemmas}
We begin to prove further regularity properties
of the solutions with initial data in $\H_\tau^1$.

\begin{lemma}
\label{thm-ex-un-wak-1}
If $z_\tau\in \H_\tau^1$ then
the (unique) solution $z(t)$ is uniformly bounded in $\H_t^1$ as $t\in[\tau,T]$,
and
$$
u \in \C([\tau, T], \h^1) \cap \C^1([\tau, T], \h).
$$
\end{lemma}

\begin{proof}
Define the energy functionals
$$L(t) = \|u(t)\|^2_{1}+\|\pt u(t)\|^2 + 2\l f(u(t))-g, Au(t)\r$$
and
$$\L(t) = L(t) + \|\eta^t\|^2_{\M_t^1}.$$
Within the Galerkin approximation scheme, we test the equation by $\varphi=A\pt u$.
This gives
$$\ddt L
+ 2\l\eta,\pt u\r_{\M^1_t} = 2\l f'(u)\pt u, Au\r.$$
Since $\|u(t)\|_1$ is uniformly bounded by Theorem~\ref{thm-ex-un-wak},
owing to \eqref{hp1-f}
we find the controls
\begin{align*}
&\|f(u)\| \leq C\big(1 + \|u\|_1^3\big)\leq C,\\
&\|f'(u)\|_{L^3(\Omega)}\leq C\big(1 + \|u\|_1^2\big)\leq C,
\end{align*}
where, along this proof, $C$ denotes a positive constant depending on the size of $z_\tau$. From
the first inequality, we easily conclude that
\begin{equation}
\label{contr-E1-Z1-n}
\frac{1}{2}\|z(t)\|_{\H_t^1}^2-C\leq \L(t)\leq
2\|z(t)\|_{\H_t^1}^2+C.
\end{equation}
In turn, from the second inequality we deduce the estimate
$$
2\l f'(u)\pt u, Au\r
\leq 2\|f'(u)\|_{L^3(\Omega)}\|\pt u\|_{L^6(\Omega)}\|Au\|
\leq C\|\pt u\|_1\|u\|_2
\leq C+C\L,
$$
so obtaining
\begin{equation}
\label{pipa}
\ddt L
+ 2\l\eta,\pt u\r_{\M^1_t} \leq C+C\L.
\end{equation}
At this point, we apply Theorem~\ref{theorem-eta-norm} for $\sigma=1$,
and we get
$$\|\eta^t\|^2_{\M_t^1}\leq \|\eta^\tau\|^2_{\M_\tau^1} + M
\int_\tau^t\|\eta^y\|^2_{\M_y^1}\d y
+2\int_\tau^t \l\eta^y,\pt u\r_{\M_y^1}\d y.$$
Adding this inequality to~\eqref{pipa} integrated in time over $[\tau,t]$ for $t\leq T$,
on account of~\eqref{contr-E1-Z1-n}, we end up with
$$\L(t) \leq  \L(\tau)+C+C\int_\tau^t \L(y)\d y.$$
Then, the Gronwall Lemma together with a subsequent application of \eqref{contr-E1-Z1-n}
yield the desired estimate
$$
\sup_{t\in[\tau,T]}\|z(t)\|_{\H_t^1}\leq C.
$$
In particular,
$$
u \in L^\infty(\tau,T;\h^2)\cap W^{1,\infty}(\tau,T;\h^1)
\subset \C([\tau,T],\h^1).
$$
Besides, paralleling Remark~\ref{remetareggy} and Subsection~\ref{SSREG}, we
learn that $\eta \in L^\infty(\tau,T;\M^1_\tau)$, and
appealing to Lemma~\ref{lemma-L1} for $\sigma=1$ we draw
by comparison
$$\ptt u\in L^1(\tau,T;\h).$$
Hence,
$$\pt u\in L^\infty(\tau,T;\h^1)\cap W^{1,1}(\tau,T;\h)\subset \C([\tau,T],\h),$$
as claimed.
\end{proof}

\begin{lemma}
\label{stima-reg}
Let $z_1(t),z_2(t)$ be two solutions.
If $z_1(\tau),z_2(\tau)\in\H_\tau^1$, then
$$
\| z_1(t) - z_2(t) \|_{\H_t}
\leq C\| z_1(\tau) - z_2(\tau) \|_{\H_\tau}
$$
for every $t \in [\tau,T]$,
where the positive constant $C$, beside $\tau$ and $T$, depends (increasingly) only on
the norms of $z_1(\tau)$ and $z_2(\tau)$ in $\H_\tau$.
\end{lemma}

\begin{proof}
We argue as in the proof of Proposition~\ref{lemma-dep}, the only
difference being that now we can use $\varphi=\pt\bar u\in \h^1$ as a test function in \eqref{eq-mem-diff}.
Accordingly, we obtain
$$\ddt(\|\bar u\|^2_1+\|\pt\bar u\|^2)
+2\int_0^\infty\mu_t(s)\l\bar\eta(s),\pt\bar u\r_1\d s
=2\l f(u_2)-f(u_1),\pt\bar u\r.$$
Leaning on \eqref{hp1-f} and exploiting the boundedness of $\|u_1\|_{1}$ and $\|u_2\|_{1}$, we estimate
\begin{align*}
2\l f(u_2)-f(u_1),\pt\bar u\r
\leq C\big(1+\|u_1\|^2_1+\|u_2\|^2_1\big)\|\bar u\|_1\|\pt\bar u\|
\leq C\big(\|\bar u\|^2_1+\|\pt\bar u\|^2\big),
\end{align*}
where $C$ depends on the size of the initial data in $\H_\tau$ only.
The conclusion follows as in the proof of Proposition~\ref{lemma-dep},
making use of Theorem~\ref{theorem-eta-norm} for
$\sigma=0$.
\end{proof}

\subsection{Approximating the solution}
Let $z_\tau\in\H_\tau$ be any fixed initial datum, and
let
$$z(t)=(u(t),\pt u(t),\eta^t)$$
be the unique solution satisfying $z(\tau)=z_\tau$.
Then, we choose a sequence $z_{\tau n}\in \H^1_\tau$ such that
$$z_{\tau n}\to z_\tau\quad\text{in } \H_\tau,$$
and we denote by
$$z_n(t)=(u_n(t),\pt u_n(t),\eta^t_n)$$
the corresponding sequence of solutions satisfying $z_n(\tau)=z_{\tau n}$.
For every $n\in\N$, we know from Lemma~\ref{thm-ex-un-wak-1} that
\begin{equation}
\label{unoZ}
u_n\in\C([\tau, T], \h^1) \cap \C^1([\tau, T], \h).
\end{equation}
Let now $t\in[\tau,T]$ be arbitrarily fixed. Proposition~\ref{lemma-dep} entails the strong convergence
\begin{equation}
\label{dueZ}
z_{n}(t) \to z(t) \quad\text{in }\H_t^{-1}.
\end{equation}
Besides, we claim that, up to a subsequence,
\begin{equation}
\label{treZ}
z_{n}(t) \xrightarrow{w} z(t) \quad\text{in }\H_t.
\end{equation}
Indeed, by Theorem~\ref{thm-ex-un-wak},
$z_n(t)$ is bounded in $\H_t$ with a bound independent of $n$
(for $z_{\tau n}$ is a bounded sequence in $\H_\tau$).
Accordingly, up to a subsequence, $z_n(t)$ has a weak limit in $\H_t$. Due to \eqref{dueZ},
such a limit equals $z(t)$.

\subsection{Conclusion of the proofs}
First, we prove the continuity of the solution in the phase space.

\begin{lemma}
The function $u$ satisfies
$$
u \in \C([\tau, T], \h^1) \cap \C^1([\tau, T], \h).
$$
\end{lemma}

\begin{proof}
It is convenient to introduce the product spaces
$${\rm W}^{-1} =\h\times\h^{-1} \and {\rm W}=\h^1\times\h,$$
and set
$$w(t)=(u(t),\pt u(t))\and w_n(t)=(u_n(t),\pt u_n(t)).$$
In light of \eqref{unoZ},
$$w_n\in \C([\tau,T],{\rm W}).$$
Besides, for every $n,m\in\N$, by Lemma~\ref{stima-reg} we have in particular the inequality
$$\|w_n(t)-w_m(t)\|_{{\rm W}}\leq C \| z_{\tau n} - z_{\tau m} \|_{\H_\tau},\quad\forall t\in[\tau,T],$$
telling that $w_n$ is a Cauchy sequence in the space
$\C([\tau,T],{\rm W})$.
Hence, it converges to some $\chi\in \C([\tau,T],{\rm W})$.
At the same time, we see from~\eqref{dueZ} that $w_n(t) \to w(t)$ in ${\rm W}^{-1}$ for every $t$,
which yields the equality $\chi=w$. Thus $w\in \C([\tau,T],{\rm W})$.
\end{proof}

This finishes the proof of Theorem~\ref{thm-ex-un}.
With a similar argument, we establish the continuous dependence estimate
of Theorem \ref{thm-cont-2}.

\begin{proof}[Proof of Theorem \ref{thm-cont-2}]
Let $z_1(t),z_2(t)\in \H_t$ be two solutions, and let $z_{1n}(t),z_{2n}(t)\in\H^1_t$ be their
respective approximating sequences.
For an arbitrarily fixed $t\in[\tau,T]$, we know from \eqref{treZ} that
$$z_{1n}(t) \xrightarrow{{\rm w}} z_1(t),\,\,z_{2n}(t) \xrightarrow{{\rm w}} z_2(t) \quad\text{in }\H_t.$$
Thus, exploiting Lemma~\ref{stima-reg} and the semicontinuity of the norm
(observe that $C$ is independent of $n$),
\begin{align*}
\|z_1(t)-z_2(t)\|_{\H_t}&\leq \liminf_{n\to\infty} \|z_{1n}(t)-z_{2n}(t)\|_{\H_t}\\
&\leq
C\,\liminf_{n\to\infty}\|z_{1n}(\tau)-z_{2n}(\tau)\|_{\H_\tau}\\
&=C\|z_1(\tau)-z_2(\tau)\|_{\H_\tau},
\end{align*}
ending the proof.
\end{proof}

\bigskip

\section*{Appendix: A Rheological Model for Aging Viscoelastic Materials}

\theoremstyle{definition}
\newtheorem{remarkAPP}{Remark}[section]
\renewcommand{\theremarkAPP}{A.\arabic{remarkAPP}}
\setcounter{equation}{0}
\setcounter{subsection}{0}
\renewcommand{\theequation}{A.\arabic{equation}}

\noindent
As already mentioned in the Introduction, the rheological model usually employed
in the description of a standard viscoelastic solid consists of
a Hookean spring and a Newtonian dashpot in series with each other (the so-called Maxwell element)
in parallel with a lone spring. Here, the idea is to reproduce the effects of the material aging via a progressive
stiffening of the spring in the Maxwell component.
This will lead to a concrete realization of equation~\eqref{eqn-mem-ht}
for a particular kernel $h_t(\cdot)$, which will be shown to comply with our assumptions {\bf (M1)}-{\bf (M4)}.

\subsection*{I. The model}
We consider axial deformations of a linear homogeneous viscoelastic body
occupying a volume $\Omega\subset\R^3$ at rest.
Since the material is homogeneous, we can represent its mechanical behavior by means of the same rheological
model at every point $x\in\Omega$. In particular, all the physical parameters turn out to be
independent of spatial coordinates.
A typical example encompassed by our analysis is a viscoelastic specimen in the form of a rectilinear rod
deforming under the action of tensile forces applied to its ends.
The aging of the material will be translated by replacing the Hooke constant
of the spring in the Maxwell element with a nondecreasing positive function.
Precisely (see fig.\ $\!$2), we denote by $\textsf{K}>0$ the rigidity of the lone spring, whereas,
concerning the Maxwell component, we denote by $\gamma>0$ the viscosity of the damper and by
$\textsf{K}_0(t)$ the rigidity of the spring at time $t$, where the function $\textsf{K}_0\in\C^1(\R)$
is supposed to be nondecreasing and to satisfy the
``initial" condition
\begin{equation}
\label{inf-K0}
\lim_{t\to -\infty}\textsf{K}_0(t)= \beta>0.
\end{equation}

\due

\begin{remarkAPP}
\label{remmyAPP}
In light of our previous discussion, the most interesting case from a physical
point of view is when
\begin{equation}
\label{inf-K0INF}
\lim_{t\to\infty}\textsf{K}_0(t)=\infty,
\end{equation}
translating the fact that the spring in the Maxwell element
becomes completely rigid in the longtime, so that
the Kelvin-Voigt viscoelastic model is recovered.
\end{remarkAPP}

\subsection*{II. The constitutive equation}
A constitutive equation is a relation between the uniaxial
strain $\epsilon=\epsilon(x,t):\Omega \times \R \to \R^3$ and the tensile stress
$\sigma=\sigma(x,t):\Omega \times \R \to \R^3$ at each point $(x,t)$.
As usual, in a rheological framework these fields are assumed to be uniform in $\Omega$, hence their dependence
on $x$ will be omitted.
With reference to fig.\ $\!$2, it is convenient to denote by $\epsilon_0(t)$ and $\epsilon_1(t)$
the strains at time $t$ of the Maxwell spring and of the damper,
respectively.
Since the true (logarithmic) strain is additive, we get
\begin{equation}
\label{eps1-eps2}
\epsilon(t) = \epsilon_0(t) + \epsilon_1(t).
\end{equation}
Besides, let $\sigma_S(t)$ be the stress of the lone spring, and
$\sigma_M(t)$ the stress of the Maxwell component.
Due to the fact that the lone spring and the Maxwell element are in parallel, we have the relation
\begin{equation}
\label{sig1-sig2}
\sigma(t) = \sigma_S(t) + \sigma_M(t).
\end{equation}
Recalling that the material is homogeneous,
we now write the constitutive equations for each of the rheological elements.
For the lone spring, the Hooke law reads
\begin{equation}
\label{eps-sigma2}
\sigma_S(t) = \textsf{K}\epsilon(t).
\end{equation}
Concerning the Maxwell element, as the Hookean spring and the Newtonian damper are in series,
they are subject to the same stress, namely,
\begin{equation}
\label{eps1-sigma1}
\sigma_M(t) = \textsf{K}_0(t)\epsilon_0(t)=\gamma\dot\epsilon_1(t),
\end{equation}
where the {\it dot} stands for derivative with respect to time.
Substituting~\eqref{eps1-sigma1} into~\eqref{eps1-eps2}, we draw
the differential identity
$$
\gamma\dot\epsilon_1(t)+\textsf{K}_0(t)\epsilon_1(t)
=\textsf{K}_0(t)\epsilon(t),
$$
which, integrated on $[r,t]$, gives
$$\epsilon_1(t)=\epsilon_1(r)\e^{-\frac1\gamma\int_0^{t-r}\textsf{K}_0(t-y)\d y}
+\frac1\gamma\int_0^{t-r}\e^{-\frac1\gamma\int_0^s \textsf{K}_0(t-y)\d y}\textsf{K}_0(t-s)\epsilon(t-s)\d s.$$
On account of~\eqref{inf-K0}, for every fixed $t\in\R$ and $p\geq 0$,
\begin{equation}
\label{alpha-nonint}
0\leq \e^{-\frac1\gamma\int_0^p \textsf{K}_0(t-y)\d y}
\leq \e^{-\frac{\beta p}{\gamma}}.
\end{equation}
Thus, under the reasonable assumption that $\epsilon_1$ is uniformly bounded in the past,
letting $r\to-\infty$ we have
$$\epsilon_1(r)\e^{-\frac1\gamma\int_0^{t-r}\textsf{K}_0(t-y)\d y}\to 0,$$
and we conclude that
$$\epsilon_1(t)=\frac1\gamma
\int_0^{\infty}\e^{-\frac1\gamma\int_0^s \textsf{K}_0(t-y)\d y}\textsf{K}_0(t-s)\epsilon(t-s)\d s.$$
On the other hand,
making use of~\eqref{eps1-eps2}-\eqref{eps1-sigma1}, we can write
$\epsilon_1$ in terms of $\epsilon$ and $\sigma$ as
$$
\epsilon_1(t)=\bigg[1+\frac{\textsf{K}}{\textsf{K}_0(t)}\bigg]\epsilon(t) - \frac{\sigma(t)}{\textsf{K}_0(t)}.
$$
Collecting the two equalities above, we end up with
\begin{equation}
\label{sigmauno}
\sigma(t)= (\textsf{K}_0(t)+\textsf{K})\epsilon(t)-\frac{1}{\gamma}\textsf{K}_0(t)
\int_0^{\infty}\e^{-\frac1\gamma\int_0^s \textsf{K}_0(t-y)\d y}\textsf{K}_0(t-s)\epsilon(t-s)\d s.
\end{equation}
At this point, an integration by parts together with a further use of~\eqref{alpha-nonint},
assuming $\epsilon$ uniformly bounded in the past,
lead to
the integral-type constitutive equation
\begin{equation}
\label{sigmadue}
\sigma(t)
= \textsf{K}\epsilon(t)+\textsf{K}_0(t)\int_0^\infty\e^{-\frac1\gamma\int_0^s \textsf{K}_0(t-y)\d y}\dot\epsilon(t-s)\d s.
\end{equation}

\subsection*{III. Mechanical evolution of the body}
The final goal is to determine the kinematic equation of the viscoelastic body.
Denoting by
$u: \Omega \times \R \to \R$ the axial displacement field relative to the reference configuration $\Omega$,
the balance of linear momentum in Lagrangian coordinates reads
$$\varrho\ptt u = \nabla\cdot \sigma + \varrho \textsf{F},$$
where $\varrho$ is the reference density of the body and $\textsf{F}$ is an external force per unit mass.
Hence, from the explicit form~\eqref{sigmadue} of $\sigma$, and recalling that
$\epsilon$ is related to the
displacement as
$\epsilon=\nabla u$,
we obtain
\begin{equation}
\label{eqpreKV}
\ptt u- k_\infty\Delta u
-\int_0^\infty k_t(s)\Delta\pt u(t-s)\d s=\textsf{F},
\end{equation}
where we set
$$
k_\infty = \frac{\textsf{K}}{\varrho}
$$
and
$$
k_t(s)= \frac{1}{\varrho}
\textsf{K}_0(t)\e^{-\frac1\gamma\int_0^s \textsf{K}_0(t-y)\d y}.
$$
Equivalently, using~\eqref{sigmauno} in place of~\eqref{sigmadue},
$$
\ptt u- h_t(0)\Delta u
-\int_0^\infty h_t'(s)\Delta u(t-s)\d s= \textsf{F},
$$
with
$$h_t(s)=k_t(s)+k_\infty.$$
The original equation~\eqref{eqn-mem-ht} is then recovered when $\textsf{F}$ is a
displacement-dependent external force of the form
$\textsf{F}=g-f(u)$.

\begin{remarkAPP}
Observe that $k_t(\cdot)$ is convex. Indeed, for every fixed $t\in\R$,
$$ k_t''(s)
=\frac{\textsf{K}_0(t)}{\varrho\gamma}\bigg[\dot{\textsf{K}}_0(t-s)+
\frac1\gamma[\textsf{K}_0(t-s)]^2\bigg]\e^{-\frac1\gamma\int_0^s \textsf{K}_0(t-y)\d y}\geq 0,$$
where we are exploiting the fact that $\textsf{K}_0$ is nondecreasing.
Besides, owing to \eqref{alpha-nonint},
$$
\int_0^\infty k_t(s)\d s
\leq \frac{1}{\varrho}\textsf{K}_0(t)
\int_0^\infty\e^{-\frac{\beta s}{\gamma}}\d s
= \frac{\gamma}{\varrho\beta}\textsf{K}_0(t),
$$
proving that $k_t(\cdot)$ is summable.
\end{remarkAPP}

\begin{remarkAPP}
In the particular case when $\textsf{K}_0(t)=\beta$ for every $t\in\R$,
we recover the classical time-independent kernel
$$k(s)=\frac{\beta}{\varrho}
\e^{-\frac{\beta s}{\gamma}}$$
widely used in the modeling of (non-aging)
standard viscoelastic solids. See e.g.\ \cite{CHR,FM,RHN}.
\end{remarkAPP}

\subsection*{IV. Verifying the assumptions on the memory kernel}
We now show that the time-dependent memory kernel
$\mu_t(\cdot)= -k'_t(\cdot)=-h'_t(\cdot)$ given by
$$\mu_t(s)=\frac{1}{\varrho\gamma}
\textsf{K}_0(t)\textsf{K}_0(t-s)\e^{-\frac1\gamma\int_0^s \textsf{K}_0(t-y)\d y}
$$
complies with assumptions {\bf (M1)}-{\bf (M4)} of Section 2.

\medskip
\noindent
$\bullet$
Assumption {\bf (M1)} is fulfilled, for $k_t(\cdot)$ is convex and summable (hence vanishing at infinity).
In particular,
$$\kappa(t)=\int_0^\infty\mu_t(s)\d s = \frac{\textsf{K}_0(t)}{\varrho}.$$

\medskip
\noindent
$\bullet$ Assumption {\bf (M2)} is fulfilled with
$$K_\tau(t)=\frac1\beta\frac{[\textsf{K}_0(t)]^2}{\textsf{K}_0(\tau)}.$$
Indeed, let $t>\tau$. Since $\textsf{K}_0$ is nondecreasing and \eqref{inf-K0} holds,
$$
\mu_t(s)
\leq \frac{1}{\varrho\gamma}
\textsf{K}_0(t)\textsf{K}_0(t-s)\e^{-\frac1\gamma\int_0^s \textsf{K}_0(\tau-y)\d y}
= \frac{\textsf{K}_0(t)\textsf{K}_0(t-s)}{\textsf{K}_0(\tau)\textsf{K}_0(\tau-s)}\mu_\tau(s)
\leq K_\tau(t)\mu_\tau(s).
$$

\medskip
\noindent
$\bullet$ Assumption {\bf (M3)} is obviously true as $\textsf{K}_0\in \C^1(\R)$.
In particular,
$$\begin{aligned}
\dot\mu_t(s)&=
\frac{1}{\varrho\gamma}\bigg[\dot{\textsf{K}}_0(t)\textsf{K}_0(t-s)+ \textsf{K}_0(t)\dot{\textsf{K}}_0(t-s)\\
\noalign{\vskip1mm}
&\qquad\quad- \frac{1}{\gamma}[\textsf{K}_0(t)]^2 \textsf{K}_0(t-s)
+\frac{1}{\gamma}\textsf{K}_0(t)[\textsf{K}_0(t-s)]^2\bigg]\e^{-\frac1\gamma\int_0^s \textsf{K}_0(t-y)\d y}.
\end{aligned}$$

\medskip
\noindent
$\bullet$ Assumption {\bf (M4)} holds with
$$M(t)=\frac{\dot{\textsf{K}}_0(t)}{\textsf{K}_0(t)}.$$
Indeed,
$$\begin{aligned}
\dot\mu_t(s)+\mu'_t(s)
&= \frac{1}{\varrho\gamma}\bigg[\dot{\textsf{K}}_0(t)-\frac1\gamma[\textsf{K}_0(t)]^2\bigg]
\textsf{K}_0(t-s)\e^{-\frac1\gamma\int_0^s \textsf{K}_0(t-y)\d y}\\
\noalign{\vskip1mm}
&\leq \frac{1}{\varrho\gamma}\dot{\textsf{K}}_0(t)\textsf{K}_0(t-s)\e^{-\frac1\gamma\int_0^s \textsf{K}_0(t-y)\d y} \\
&= \frac{\dot{\textsf{K}}_0(t)}{\textsf{K}_0(t)}\mu_t(s).
\end{aligned}$$

\subsection*{V. Recovering Kelvin-Voigt}
The aim of this final subsection is to render Remark~\ref{remmyAPP} more rigorous.
Namely, we prove that within~\eqref{inf-K0}-\eqref{inf-K0INF}
the distributional convergence
$$k_t\to\frac{\gamma}{\varrho}\delta_0$$
generically occurs as $t\to\infty$, so that the equation with memory~\eqref{eqpreKV} collapses
into the Kelvin-Voigt viscoelastic model
$$\ptt u- k_\infty\Delta u
-\frac{\gamma}{\varrho}\Delta\pt u=\textsf{F}.
$$
More precisely, this will happen under the additional very mild assumption
\begin{equation}
\label{K0uno}
\lim_{t\to\infty}\frac{\dot{\textsf{K}}_0(t)}{[\textsf{K}_0(t)]^2}=0.
\end{equation}
This is always the case,
for instance, when $\textsf{K}_0$ is eventually concave down as $t\to\infty$.
Since the function $k_t(\cdot)$ is nonnegative for every $t$, our claim follows by showing that,
for every fixed $\nu\geq 0$,
$$\lim_{t\to\infty}\int_\nu^\infty k_t(s)\d s=
\begin{cases}
\gamma/\varrho &\text{if }\nu=0,\\
0 &\text{if }\nu>0.
\end{cases}
$$
To this end, introducing the antiderivative
$$
\textsf{H}(t)=\int_0^t \textsf{K}_0(y)\d y,
$$
let us write
$$
k_t(s)= \frac{1}{\varrho}
\textsf{K}_0(t)\e^{-\frac1\gamma\textsf{H}(t)}\e^{\frac1\gamma\textsf{H}(t-s)}
$$
and denote, for $t\geq \nu$,
$${\mathfrak I}_1(t)=\int_t^\infty k_t(s)\d s
\and
{\mathfrak I}_2(t)=\int_\nu^t k_t(s)\d s.$$

\smallskip
\noindent
$\bullet$ We first establish the convergence ${\mathfrak I}_1(t)\to 0$. Indeed, for $s\geq t$, we infer from \eqref{inf-K0}
that
$$\textsf{H}(t-s)\leq \beta (t-s).
$$
Accordingly,
$$\int_t^\infty \e^{\frac1\gamma\textsf{H}(t-s)}\d s
\leq \frac\gamma\beta,$$
which readily gives
$${\mathfrak I}_1(t)\leq  \frac{\gamma}{\varrho\beta}
\textsf{K}_0(t)\e^{-\frac1\gamma\textsf{H}(t)}.
$$
In order to reach the desired conclusion, we note that~\eqref{K0uno} implies that the nonnegative function
$$\textsf{Q}(t)=\textsf{K}_0(t)\e^{-\frac1\gamma\textsf{H}(t)}$$
is eventually decreasing, hence bounded at infinity.
Therefore, de l'H\^opital's rule and a further exploitation of~\eqref{K0uno} give
$$\lim_{t\to\infty}\textsf{Q}(t)=
\gamma \lim_{t\to\infty}
\frac {\dot{\textsf{K}}_0(t)}{\textsf{K}_0(t)\e^{\frac1\gamma\textsf{H}(t)}}
=\gamma \lim_{t\to\infty}
\frac {\dot{\textsf{K}}_0(t)}{[\textsf{K}_0(t)]^2}\,\textsf{Q}(t)=0.
$$

\smallskip
\noindent
$\bullet$ As far as ${\mathfrak I}_2(t)$ is concerned, we write
$${\mathfrak I}_2(t)=\frac{1}{\varrho}
\textsf{Q}(t)\int_0^{t-\nu} \e^{\frac1\gamma\textsf{H}(y)}\d y.
$$
Then, since we showed that $\textsf{Q}(t)\to 0$,
applying de l'H\^opital's rule and exploiting~\eqref{K0uno} we get
$$\lim_{t\to\infty}{\mathfrak I}_2(t)
=\frac1\varrho \lim_{t\to\infty}
\frac {\int_0^{t-\nu} \e^{\frac1\gamma\textsf{H}(y)}\d y}{\frac1{\textsf{Q}(t)}}
=\frac1\varrho \lim_{t\to\infty}
\frac {\e^{-\frac1\gamma [\textsf{H}(t)-\textsf{H}(t-\nu)]}}{\frac1\gamma-\frac{\dot{\textsf{K}}_0(t)}{[\textsf{K}_0(t)]^2}}
=\frac\gamma\varrho \lim_{t\to\infty}
\e^{-\frac1\gamma [\textsf{H}(t)-\textsf{H}(t-\nu)]}.
$$
The latter limit clearly equals $1$ when $\nu=0$, whereas when $\nu>0$
$$
\e^{-\frac1\gamma [\textsf{H}(t)-\textsf{H}(t-\nu)]}\to 0.$$
Indeed, recalling \eqref{inf-K0INF} and the fact that $\textsf{K}_0$ is nondecreasing,
$$\textsf{H}(t)-\textsf{H}(t-\nu)=\int_{t-\nu}^t \textsf{K}_0(y)\d y\geq \nu \textsf{K}_0(t-\nu)\to\infty.
$$
The claim is proven.

\begin{remarkAPP}
We point out that the function $k_t(\cdot)$ has an independent interest. Indeed, for $\varrho=\gamma$, it
provides an approximation (from the right) of the Dirac delta function, which does not seem to be
known in the literature.
\end{remarkAPP}



\end{document}